\documentclass{amsart}

\usepackage{amsmath,amssymb}
\usepackage{mathrsfs}
\usepackage{mathtools}
\usepackage{stmaryrd}
\usepackage{enumerate}
\usepackage{tikz-cd}
\usepackage[all]{xy}
\usepackage{aliascnt}
\usepackage[colorlinks=true, linkcolor=blue, citecolor=blue]{hyperref}

\usepackage{enumitem}

\usepackage{combelow}

\newtheorem{theorem}{Theorem}

\newaliascnt{lemma}{theorem}
\newtheorem{lemma}[lemma]{Lemma}
\aliascntresetthe{lemma}

\newaliascnt{corollary}{theorem}

\aliascntresetthe{corollary}

\newaliascnt{proposition}{theorem}
\newtheorem{proposition}[proposition]{Proposition}
\aliascntresetthe{proposition}

\newaliascnt{conjecture}{theorem}

\aliascntresetthe{conjecture}

\newaliascnt{question}{theorem}

\aliascntresetthe{question}

\theoremstyle{definition}

\newaliascnt{definition}{theorem}
\newtheorem{definition}[definition]{Definition}
\aliascntresetthe{definition}

\newaliascnt{remark}{theorem}
\newtheorem{remark}[remark]{Remark}
\aliascntresetthe{remark}

\newaliascnt{example}{theorem}
\newtheorem{example}[example]{Example}
\aliascntresetthe{example}

\newaliascnt{exampledef}{theorem}
\newtheorem{exampledef}[exampledef]{Example-Definition}
\aliascntresetthe{exampledef}

\newaliascnt{notation}{theorem}
\newtheorem{notation}[notation]{Notation}
\aliascntresetthe{notation}

\newtheorem*{acknowledgements}{Acknowledgements}

\newif\ifhascomments \hascommentstrue
\ifhascomments
  \newcommand{\matt}[1]{{\color{red}[[\ensuremath{\spadesuit\spadesuit\spadesuit} #1]]}}
  \newcommand{\jeremy}[1]{{\color{red}[[\ensuremath{\clubsuit\clubsuit\clubsuit} #1]]}}
\else
  \newcommand{\matt}[1]{}
  \newcommand{\jeremy}[1]{}
\fi

\renewcommand{\setminus}{\smallsetminus}

\newcommand{\Z}{\mathbb{Z}}

\newcommand{\QQ}{\mathbb{Q}}

\newcommand{\cT}{\mathscr{T}}
\newcommand{\cX}{\mathcal{X}}
\newcommand{\cY}{\mathcal{Y}}
\newcommand{\cZ}{\mathcal{Z}}
\newcommand{\cW}{\mathscr{W}\!}
\newcommand{\cI}{\mathcal{I}}

\newcommand{\cC}{\mathcal{C}}
\newcommand{\cD}{\mathcal{D}}
\newcommand{\cU}{\mathcal{U}}
\newcommand{\cO}{\mathcal{O}}
\newcommand{\cV}{\mathcal{V}}
\newcommand{\cF}{\mathcal{F}}

\newcommand{\cH}{\mathcal{H}}
\newcommand{\cJ}{\mathscr{J}}

\newcommand{\cM}{\mathcal{M}}

\newcommand{\sL}{\mathscr{L}}

\newcommand{\bL}{\mathbb{L}}

\newcommand{\bG}{\mathbb{G}}
\newcommand{\bA}{\mathbb{A}}

\newcommand{\fm}{\mathfrak{m}}

\newcommand{\diff}{\mathrm{d}}
\newcommand{\red}{\mathrm{red}}
\newcommand{\id}{\mathrm{id}}
\newcommand{\Gor}{\mathrm{Gor}}

\newcommand{\g}{\mathfrak{g}}

\DeclareMathOperator{\codim}{codim}
\DeclareMathOperator{\s}{s}

\DeclareMathOperator{\can}{can}

\DeclareMathOperator{\het}{ht}

\DeclareMathOperator{\Spec}{Spec}
\DeclareMathOperator{\Ext}{Ext}
\DeclareMathOperator{\ord}{ord}
\DeclareMathOperator{\Hom}{Hom}
\DeclareMathOperator{\Isom}{Isom}

\DeclareMathOperator{\GL}{GL}

\DeclareMathOperator{\coh}{coh}
\DeclareMathOperator{\sm}{sm}

\DeclareMathOperator{\SL}{SL}

\DeclareMathOperator{\uHom}{\underline{\Hom}}
\DeclareMathOperator{\uIsom}{\underline{\Isom}}
\DeclareMathOperator{\uSpec}{\underline{\Spec}}

\tikzset{cong/.style={draw=none,edge node={node [sloped, allow upside down, auto=false]{$\cong$}}},
         Isom/.style={above,every to/.append style={edge node={node [sloped, allow upside down, auto=false]{$\sim$}}}}}

\title{Beyond twisted arcs:~a McKay correspondence for reductive groups}

\author{Matthew Satriano and Jeremy Usatine}

\thanks{MS was partially supported by a Discovery Grant from the National Science and Engineering Research Council of Canada as well as a Mathematics Faculty Research Chair from the University of Waterloo}

\address{Matthew Satriano, Department of Pure Mathematics, University of Waterloo}
\email{msatriano@uwaterloo.ca}

\address{Jeremy Usatine, Department of Mathematics, Florida State University}
\email{jusatine@fsu.edu}

\begin{document}

\begin{abstract}
We introduce a natural generalization of twisted maps, called \emph{warped maps}. While twisted maps play an important role in the study of Deligne--Mumford stacks, warped maps are better suited for studying Artin stacks. Heuristically, warped maps see the hidden proper-like behavior satisfied by good moduli space maps. Specifically, we show that every arc of a good moduli space admits a \emph{canonical} lift, in a warped sense, thereby proving a valuative criterion for good moduli spaces. Furthermore, we prove that warped maps to an Artin stack $\mathcal{X}$ are given by usual maps to an auxiliary Artin stack $\mathscr{W}(\mathcal{X})$, immediately obtaining a versatile framework for bootstrapping results about usual maps to the setting of warped maps.  As an application we obtain a motivic change of variables formula which, given a stacky resolution of singularities $\mathcal{X} \to Y$, canonically expresses any given motivic integral over arcs of $Y$ as a certain motivic integral over warped arcs of $\mathcal{X}$. 
In particular, this yields a McKay correspondence for linearly reductive groups.
\end{abstract}

\maketitle

\numberwithin{theorem}{section}
\numberwithin{lemma}{section}
\numberwithin{corollary}{section}
\numberwithin{proposition}{section}
\numberwithin{conjecture}{section}
\numberwithin{question}{section}
\numberwithin{remark}{section}
\numberwithin{definition}{section}
\numberwithin{example}{section}
\numberwithin{exampledef}{section}
\numberwithin{notation}{section}

\setcounter{tocdepth}{1}

\tableofcontents









\section{Introduction}

%
%

If $Y$ is a log-terminal variety with quotient singularities, its stringy Hodge numbers have a beautiful interpretation as orbifold Hodge numbers of a smooth Deligne--Mumford stack $\cX$ where $\cX\to Y$ is a canonical crepant resolution. There are two major obstructions to finding a similar interpretation for the stringy Hodge numbers when $Y$ has worse than quotient singularities. The first is that $Y$ rarely admits a crepant resolution by a Deligne--Mumford stack. In \cite{SatrianoUsatine3}, we were able to overcome this obstruction by constructing crepant resolutions for all log-terminal varieties using Artin stacks. This leads to the second major obstruction:~Artin stacks (other than Deligne--Mumford stacks) are not separated, creating major difficulties in applying previous techniques in motivic integration. In this paper we introduce a versatile framework that overcomes this second obstruction.

\subsection{Overview}

Moduli spaces of twisted curves and twisted arcs are of fundamental importance in algebraic geometry with far-reaching applications in Gromov--Witten theory, birational geometry, mirror symmetry, and motivic integration. First defined by Abramovich and Vistoli \cite{AbramovichVistoli}, twisted curves were introduced to compactify the space of stable maps. Their work played a vital role in Chen and Ruan's work on orbifold cohomology \cite{ChenRuan} and the Crepant Resolution Conjecture \cite{Ruan06}, leading to a flurry of activity in the subject \cite{BG,CCIT,CCITtoric,CIJ}. Twisted curves have even seen applications in number theory, where they were used to 
give point-counting heuristics on stacks, unifying the Batyrev--Manin and Malle conjectures; see the work of the first author, Ellenberg, and Zureick--Brown \cite{ESZB}, as well as Darda--Yasuda \cite{DardaYasuda}. Inspired by \cite{AbramovichVistoli}, Yasuda \cite{Yasuda2004} introduced a local version of twisted curves, called twisted arcs. 
Using moduli spaces of twisted arcs, he generalized the motivic McKay correspondence \cite{Yasuda2004,Yasuda2006, Yasuda2019}, relating stringy Hodge numbers of quotient singularities to the orbifold cohomology of Deligne--Mumford stacks. This opened the door to the use of stack-theoretic techniques in the study of motivic integration.


Due to the influential work of many authors, it has become increasingly apparant that for many geometric applications, one must consider Artin stacks which are not Deligne--Mumford. For example, Artin stacks were instrumental in the groundbreaking work of Halpern--Leistner \cite{HalpernLeistner} and Ballard--Favero--Katzarkov \cite{BallardFaveroKatzarkov} on the Bondal--Orlov Conjecture \cite{BO} (derived McKay correspondence). Artin stacks have  also been essential in functorial resolution of singularities algorithms \cite{AbramovichTemkinWlodarczyk,AbramovichTemkinWlodarczyk2,AbramovichQuekLogRes} as well as applications to the monodromy conjecture \cite{Quekmonodromy}. 


Furthermore from the point of view of the motivic McKay correspondence, any progress beyond quotients by finite groups necessitates the use of Artin stacks. Indeed, Yasuda computed the stringy Hodge numbers of varieties with quotient singularities by considering crepant resolutions $\cX\to Y$ which are coarse space maps and where $\cX$ is a smooth Deligne--Mumford stack. However, outside of the quotient singularities case, $Y$ is \emph{never} the coarse space of a smooth Deligne--Mumford stack. Worse, even when one drops the assumption that $\cX\to Y$ is a coarse space map, $Y$ rarely admits a crepant resolution by a Deligne--Mumford stack at all. A natural remedy is to allow for resolutions by Artin stacks. This of course raised the question of when such resolutions exist.

\vspace{0.8em}

While crepant resolutions by varieties or Deligne--Mumford stacks are rare, we proved in recent work \cite{SatrianoUsatine3} that crepant resolutions by Artin stacks hold in great generality:~every log-terminal variety $Y$ admits a crepant resolution by a smooth Artin stack $\cX$. Since log-terminal singularities are precisely the class for which Batyrev defined stringy Hodge numbers  \cite{Batyrev1998}, this shows that resolutions by smooth Artin stacks provide a natural setting by which to study stringy invariants. Furthermore, log-terminal singularities are a natural class arising in the context of the minimal model program (MMP). Since motivic integration and the MMP are largely controlled by maps from discs and curves, for applications to birational geometry and motivic McKay correspondences, it is particularly important to understand how maps from curves or arcs on $Y$ lift to twisted versions on $\cX$. One is therefore in need of new tools for lifting curves and arcs via resolutions by Artin stacks.

In this paper, we develop such tools, giving a counterpart to twisted maps (called \emph{warped maps}), where one allows the target $\cX$ to be an Artin stack (see Definition \ref{def:Artin-arc}). We stress that even when $\cX$ is Deligne--Mumford, our theory is new (see later in the introduction for further explanation). One of the major hurdles our theory of warped maps allows us to overcome is the fact that Artin stacks (which are not Deligne-Mumford) are rarely separated. Indeed, for Deligne--Mumford stacks, the definition of twisted curves is forced by the valuative criterion:~if $\cX$ is Deligne--Mumford and $\pi\colon\cX\to Y$ is its coarse space, properness of $\pi$ and the valuative criterion tells one precisely how to define twisted curves if one wishes to lift curves on $Y$ to twisted curves on $\cX$. On the other hand, when $\cX$ is Artin, lack of separatedness means one needs a completely different approach.

In fact, our theory of warped maps can be seen as providing a valuative criterion satisfied by good moduli space maps. Good moduli space maps feature prominently in the study of Artin stacks, and while not literally proper, work of Alper \cite{Alper}, Alper--Hall--Rydh \cite{AHRLuna,AHRetalelocal}, and Alper--Halpern-Leistner--Heinloth \cite{AHLH} have shown that they behave in many respects like proper coarse space maps:~they 
are universally closed, satisfy formal GAGA and a cohomology and base change theorem, as well as have \'etale local structure similar that of coarse spaces. This has suggested that good moduli space maps $\cX\to Y$ should satisfy a valuative criterion analogous to that of proper maps; in Theorem \ref{thm:val-crit}, we prove that arcs of $Y$ admit \emph{canonical} warped lifts to $\cX$, thereby yielding a new valuative criterion.

Furthermore, we provide a versatile framework for working with warped maps. Indeed, while the valuative criterion and the lack of properness for good moduli spaces provides the impetus for the construction of warped maps, the most prominent feature of our paper is the construction of an Artin stack $\cW(\cX)$ which \emph{parameterizes} warped maps, i.e., a warped map from $T$ to $\cX$ is given by a usual map $T\to\cW(\cX)$, see Theorem \ref{thm:main--A0-algebraic}. Thus, this allows one to bootstrap results from usual maps to warped maps. This bootstrapping technique is essential in our main applications to motivic integration and the McKay correspondence (Theorems \ref{thm:mcvf-canonical-intro} and \ref{thm:reductive-McKay}), and is analogous to Olsson's work in logarithmic geometry \cite{OlssonLogStack}, where he showed that logarithmic morphisms $S\to Z$ are the same as usual morphisms $S\to\mathcal{L}og_Z$ to an auxiliary stack, and hence properties of log maps are inherited from properties of usual maps to stacks.

\vspace{0.8em}

With the theory of warped maps in hand, we prove a motivic change of variables formula for Artin stacks (see Theorem \ref{thm:mcvf-canonical-intro}) and give a McKay correspondence for linearly reductive groups (see Theorem \ref{thm:reductive-McKay}). Our change of variables formula for Artin stacks applies, in particular, to crepant resolutions of all log-terminal singularities, as constructed in our previous work \cite{SatrianoUsatine3}. While other motivic change of variables formulas for stacks have appeared in the literature, they are all either in the Deligne--Mumford setting or involve making choices, e.g., choosing lifts to arcs of $\cX$. Such choices do not always exist and even when they do, there can be infinitely many such choices. On the other hand, our motivic change of variables formula here is \emph{completely canonical}:~by our valuative criterion (Theorem \ref{thm:val-crit}) outside of a set of measure zero, any arc $\varphi$ of $Y$ lifts canonically to a warped arc $\varphi^{\can}$ of $\cX$. Our motivic change of variables formula then expresses any motivic integral over arcs $\varphi$ of $Y$ as an explicit motivic integral over the warped arcs $\varphi^{\can}$ of $\cX$.

We wish to emphasize that in addition to our motivic change of variables formula being a novel result, the proof technique is unlike any that previously existed, even when one restricts to the case of Deligne--Mumford stacks. Indeed, for Yasuda to construct a motivic measure for twisted arcs, he needed to first construct spaces parametrizing twisted jets and arcs and then get enough explicit control of these spaces in order to construct the measure. In comparison, our main theorem (Theorem \ref{thm:main--A0-algebraic}) showing $\cW(\cX)$ is an Artin stack allows us to completely sidestep the need to construct a motivic measure for warped arcs:~since a warped arc of $\cX$ is the same as a usual arc of $\cW(\cX)$, the motivic measure on warped arcs is inherited (via bootstrapping) from the motivic measure for usual (untwisted) arcs on stacks. Furthermore via this bootstrapping, the change of variables formula in \cite{SatrianoUsatine2, SatrianoUsatine5} is automatically upgraded to the warped setting (see Theorem \ref{thm:mcvf-canonical-intro}). In other words, our bootstrapping technique (Theorem \ref{thm:main--A0-algebraic}) allows us to sidestep the need to prove a motivic change of variables formula in the warped setting from scratch.

\subsection{Definitions:~warped maps}

For the remainder of the introduction, we fix the following set-up.

\begin{notation}\label{not:main}
Let $\cX$ be a locally finite type Artin stack with affine diagonal over an algebraically closed field $k$ of characteristic $0$.
\end{notation}

We come to the central definition of the paper.

\begin{definition}\label{def:Artin-disc}
A \emph{warp\footnote{The name ``warp'' was chosen as a variant on ``twist'' that is evocative of science fiction.} of} a scheme $T$ is a flat, finitely presented, good moduli space map $\pi\colon\cT\to T$ with affine diagonal.
\end{definition}

\begin{definition}\label{def:Artin-arc}
A \emph{warped map} from $T$ to $\cX$ is a pair $(\pi\colon\cT\to T,f\colon\cT\to\cX)$ with $\pi$ a warp of $T$ and $f$ representable. We denote such a warped map by $T\rightsquigarrow\cX$. Warped maps naturally form a category fibered in groupoids
\[
\cW(\cX)\to(\textrm{\textbf{Sch}}/k),
\]
where $\cW(\cX)(T)$ is the category of warped maps from $T$ to $\cX$.
\end{definition}

\begin{remark}\label{rmk:2-stack-1-stack}
$\cW(\cX)$ is most naturally viewed as a $2$-stack, however it is equivalent to a $1$-stack as follows. A morphism from $(\pi'\colon\cT'\to T',f'\colon\cT'\to\cX)$ to $(\pi\colon\cT\to T,f\colon\cT\to\cX)$ over $g\colon T'\to T$ is given by an equivalence class $(\phi,\alpha)$ where $\phi\colon\cT'\to\cT$ satisfies $\pi\phi=g\pi'$ and induces an isomorphism $\cT'\to\cT\times_T T'$, and $\alpha\colon f\phi\Rightarrow f'$ is a $2$-arrow. Two such pairs $(\phi,\alpha)$ and $(\phi',\alpha')$ are equivalent if there exists a $2$-arrow $\beta\colon\phi'\Rightarrow\phi$ such that $\alpha'=\alpha\beta$. Since $f$ is representable, such a $\beta$ is unique if it exists.
\end{remark}

\begin{example}
Twisted stable maps (resp.~twisted arcs) are examples of warped maps where $T$ is a curve (resp.~a disc).
\end{example}

\begin{example}
A usual map $T\to\cX$ can be thought of as the warped map $(\id\colon T\to T,f\colon T\to\cX)$. This yields a morphism
\[
\tau\colon\cX\to\cW(\cX).
\]
\end{example}

\begin{definition}\label{def:warped-disc}
A \emph{warped disc} is a warp $\cD\to D$ where $D$ is a disc (i.e., $D=\Spec k'[[t]]$ with $k'$ a field). A \emph{warped arc} is a warped map from $D$ to $\cX$.
\end{definition}

The following arcs play a special role:~they are precisely the output of our valuative criterion Theorem \ref{thm:val-crit}

\begin{exampledef}[{Canonical warped lifts}]\label{def:can-lift-toArtin-arc}
In this example, we have a good moduli space $p\colon\cX\to Y$ and let $U\subset Y$ be the largest open subspace over which $p$ is an isomorphism. Let $k'/k$ be a field extension, $D=\Spec k'[[t]]$, and $\varphi\colon D\to Y$ where the generic point of $D$ maps into $U$; said more succinctly, $\varphi\in|\sL(Y)\setminus\sL(Y\setminus U)|$. Then the pullback map $p'\colon\cX\times_Y D\to D$ is an isomorphism over the generic point and hence, there is a unique irreducible component $\cZ\subset \cX\times_Y D$ containing the generic fiber. Consider the commutative diagram
\begin{equation}\label{eqn:phican}
\xymatrix{
\cD\ar[r]^-{q}\ar[dr]_-{\pi} & \cX\times_Y D\ar[r]^-{\psi}\ar[d]^-{p'} & \cX\ar[d]^-{p}\\
& D\ar[r]^-{\varphi} & Y
}
\end{equation}
where $q$ is the normalization of $\cZ$. Let $f=\psi q$. We refer to the pair
\begin{equation}\label{eqn:phicaneeqn}
\varphi^{\can}:=(\pi\colon\cD\to D,f\colon\cD\to\cX)
\end{equation}
as the \emph{canonical lift of} $\varphi$.
\end{exampledef}

The canonical lift is \emph{universal} in the sense that every twisted arc $\cD'\to\cX$ lifting $\varphi$ factors uniquely through $\varphi^{\can}$; this follows from the universal property of normalization, see Proposition \ref{prop:twistede-discs-factor-through-can-lift-toArtin-arc}. Furthermore, in the case where $\cX$ is a Deligne--Mumford stack with finite diagonal (or more generally whenever $p\colon\cX\to Y$ is separated), canonical arcs precisely recover twisted arcs of $\cX$; see Proposition \ref{prop:recover-twisted arcs}.


\subsection{Statement of results}

Our first theorem is that $\cW(\cX)$ is an Artin stack. Furthermore, it comes equipped with a ``main component'' corresponding to the closure of the map $\tau\colon\cX\to\cW(\cX)$ defined above.

\begin{theorem}\label{thm:main--A0-algebraic}
Using Notation \ref{not:main}, 
\begin{enumerate}
\item\label{thm:A0-algebraic} $\cW(\cX)$ is a locally finite type Artin stack over $k$, 
\item\label{prop:X->A0X-open immersion} $\tau\colon\cX\to\cW(\cX)$ is an open immersion, and
\item\label{thm:gms->A0XsepDiag-aff-geom-stab} 
$\cW(\cX)$ has separated diagonal and affine geometric stabilizers if $\cX$ admits a scheme as a good moduli space.
\end{enumerate}
\end{theorem}


\begin{remark}
In fact, we prove a stronger version of Theorem \ref{thm:main--A0-algebraic}(\ref{thm:gms->A0XsepDiag-aff-geom-stab}):~the map $I_{\cW(\cX)}\to\cW(\cX)$ from the inertia stack is quasi-affine; 
see Theorem \ref{thm:sep-diag-coh-affine}.
\end{remark}

The key upshot of \autoref{thm:main--A0-algebraic} is that warped maps of $\cX$ are actually the same as usual maps of the Artin stack $\cW(\cX)$. Therefore results that apply to usual maps immediately apply to warped maps as well. In other words, we obtain a versatile framework for bootstrapping that should be broadly applicable to the study of twisted and warped maps. In particular, we immediately obtain a theory of motivic integration over warped arcs as a consequence of the theory of motivic integration over usual arcs defined in \cite{SatrianoUsatine}. Furthermore, our Artin stack $\cW(\cX)$ is new even when $\cX$ is Deligne--Mumford. For example, since we allow the source of our warps to vary in moduli, $\cW(\cX)$ provides new objects that interpolate between Yasuda's twisted arcs.



It is straightforward from the definitions and \cite{AHRetalelocal} to show that if $p\colon\cX\to Y$ is any morphism to an algebraic space (e.g., a good moduli space map), then there is a natural map 
\[
\underline{p}\colon\cW(\cX)\to Y,
\]
see Section \ref{sec:functoriality}. By \autoref{thm:main--A0-algebraic}, we may immediately apply the motivic change of variables formula from \cite{SatrianoUsatine2, SatrianoUsatine5} to $\underline{p}$. We now explain why the change of variables formula applied to $\underline{p}$, as opposed to $p$, behaves particularly nicely in light of the non-properness of $\cX$.

The map $\underline{p}$ induces a morphism of arc stacks
\begin{equation}\label{eqn:LA0p-structure-map}
\sL(\cW(\cX))\to \sL(Y),
\end{equation}
which at the level $k'$-points is given as follows. Letting $D=\Spec k'[[t]]$, the map sends a warped arc $(\pi\colon\cD\to D,f\colon\cD\to\cX)$ to the unique map $\overline{f}$ making the diagram
\[
\xymatrix{
\cD\ar[r]^-{f}\ar[d]_-{\pi} & \cX\ar[d]^-{p}\\
D\ar[r]^-{\overline{f}} & Y
}
\]
commute. 

Our valuative criterion shows that, after discarding a measure 0 subset of $\sL(Y)$, the map \eqref{eqn:LA0p-structure-map} induces a bijection from a canonically defined subset of $\sL(\cW(\cX))$. To state this precisely, we isolate a key feature shared by all twisted arcs to a Deligne--Mumford stack. 
This feature that we require is a certain finiteness property relative to the good moduli space.

\begin{definition}\label{def:integrity}
Let $f\colon\cY\to\cZ$ be a morphism of Artin stacks admitting good moduli spaces $\cY\to Y$ and $\cZ\to Z$. This yields a map $\overline{f}\colon Y\to Z$ by \cite[Theorem 3.12]{AHRetalelocal}. We say $f$ \emph{has integrity} if the induced map $f^\sharp\colon\cY\to \cZ\times_Z Y$ is finite. If $\cX$ has a good moduli space, then a warped map $(\cT\to T,f\colon\cT\to\cX)\in\cW(\cX)(T)$ \emph{has integrity} if $f$ does.\footnote{The name \emph{integrity} derives from the fact that $f$ is close to being the pullback of a map on good moduli spaces, which we think of as the ``underlying goodness.''}
\end{definition}

\begin{remark}
\label{rmk:integrity-basic-properties}
We prove in Proposition \ref{prop:openness-geom-fibers-for-integrity} that integrity is an open condition which may be checked on geometric fibers.
\end{remark}

For the rest of the introduction, we additionally fix the following set-up. Recall from \cite[Definition 1.3]{SatrianoUsatine3} that a map $p\colon\cX\to Y$ to an algebraic space is called \emph{strongly birational} if there is a dense open subspace $U\subset Y$ over which $p$ is an isomorphism. We refer to the largest such $U$ as the \emph{strong birationality locus}.

\begin{notation}\label{not:main-gms}
In addition to Notation \ref{not:main}, assume $p\colon\cX\to Y$ is a strongly birational good moduli space. Let $U\subset Y$ be the strong birationality locus and $\cU=p^{-1}(U)$.
\end{notation}

\begin{definition}
Using Notation \ref{not:main-gms}, we let the \emph{canonical locus} $\cC_\cX\subset|\sL(\cW(\cX))|$ be the set of isomorphism classes of warped arcs $(\pi\colon\cD\to D,f\colon\cD\to\cX)$ such that
\begin{enumerate}[label=(\alph*)]
\item $\cD$ is normal, 
\item $f$ maps the generic fiber of $\pi$ into $\cU$, and
\item $f$ has integrity.
\end{enumerate}
\end{definition}

In Proposition \ref{prop:can-lift-toArtin-arc} and Remark \ref{rmk:can-lift-toArtin-arc}, we prove that 
canonical lifts $\varphi^{\can}$ all live in $\cC_\cX$. The content of our valuative criterion is that $\varphi^{\can}$ is the only lift of $\varphi$ in $\cC_\cX$:

\begin{theorem}[{Valuative Criterion}]
\label{thm:val-crit}
Using Notation \ref{not:main-gms}, for every field $k'$, the map $$\cC_\cX(k')\xrightarrow{\simeq}|\sL(Y)\setminus\sL(Y\setminus U)|(k')$$ induced by \eqref{eqn:LA0p-structure-map} is bijective on isomorphism classes of $k'$-points.

In other words, if $\varphi\colon\Spec k'[[t]]\to Y$ is an arc mapping the generic point to $U$, then $\varphi^{\can}\in\cC_\cX(k')$ is the unique lift of $\varphi$.
\end{theorem}

\subsection{The McKay correspondence}

As an application of our theory of warped maps, we obtain a new motivic change of variables formula for stacks. This yields a McKay correspondence for linearly reductive groups. As we have noted earlier, arcs of $Y$ may admit many lifts to (twisted) arcs of $\cX$. As a result, any motivic change of variables formula for $p\colon\cX\to Y$ must involve choices; namely given any measurable set $\cC$ of arcs of $Y$, one must choose a measurable set of arcs of $\cX$ mapping bijectively to $\cC$. Using warped maps, however, we have a completely canonical set of warped arcs of $\cX$. Our motivic change of variables formula relates a motivic integral over $\cC$ to a motivic integral over the set of canonical lifts.

In order for these motivic integrals to be well-defined, we need that $\cC_\cX$ is essentially given by a cylinder, i.e., we need normality of $\cD$ and integrity of $f$ to be properties that are determined by a constructible condition on a fixed truncation of our warped arc. This fact, which is not at all obvious, is the content of the following result.

\begin{theorem}\label{thm:main-canonical-cylinder}
There exists a cylinder whose intersection with $\sL(\cX)\setminus\sL(\cX\setminus\cU)$ agrees with $\cC_\cX$.
\end{theorem}

\begin{remark}
While the definiton of the cylinder in Theorem \ref{thm:main-canonical-cylinder} is too technical to state in the introduction, we mention that it is defined by constructible sets in $\sL_1(\cW(\cX))$. This cylinder is constructed and characterized in Definition \ref{def:can-cylinder} and Theorem \ref{thm:main-can-cylinder}.
\end{remark}

We emphasize again a key conceptual ingredient leading to our motivic change of variables formula Theorem \ref{thm:mcvf-canonical-intro} (see also Theorem \ref{thm:mcvf-canonical}):~since Theorem \ref{thm:main--A0-algebraic} implies that warped maps from $T$ to $\cX$ are the same as usual maps $T\to\cW(\cX)$, we may immediately bootstrap our motivic change of variables formula in \cite{SatrianoUsatine2, SatrianoUsatine5} to warped setting.


\begin{theorem}\label{thm:mcvf-canonical-intro}
Using Notation \ref{not:main-gms}, assume $Y$ is an irreducible scheme and $U$ is smooth. For any measurable function $f\colon\sL(Y) \to \Z \cup \{\infty\}$ with $\bL^f$ integrable on $\sL(Y)$, we have
\[
	 \int_{\sL(Y)} \bL^{f} \diff\mu_Y \ =\  \int_{\cC_\cX} \bL^{f \circ \sL(\underline{p}) - \het_{\cW(\cX)/Y}} \diff\mu_{\cW(\cX), \dim Y}.
\]
\end{theorem}


\begin{remark}
For simplicity, we stated Theorems \ref{thm:mcvf-canonical-intro} and \ref{thm:mcvf-canonical} where $\cX \to Y$ is a good moduli space map. However, it applies equally well to any stacky resolution given by the composition of a good moduli space map and a proper map; indeed one simply combines the above theorem with the usual valuative criterion for proper maps. We recall that by \cite{SatrianoUsatine3}, if $Y$ has log-terminal singularities, such a resolution can always be chosen to be crepant. Thus any motivic integral over arcs of a log-terminal variety can be expressed canonically as a motivic integral over warped arcs of a \emph{crepant} stacky resolution of singularities.
\end{remark}


\vspace{0.6em}

We end the introduction by giving the McKay correspondence for linearly reductive groups. 
Batyrev's McKay correspondence \cite{Batyrev99} (conjectured by Reid \cite{Reid}) is given for quotients $\bA^n/G$ where $G$ is a finite subgroup of $\SL_n$. For algebraic groups $G$, we need an additional constraint which appears as well in the work of \v{S}penko--van den Bergh \cite{SpenkoVanDenBergh17} on non-commutative crepant resolutions (NCCRs).

To state our McKay correspondence, we must first recall some terminology. If $Y$ is log-terminal, let $m$ be such that $Y$ is $m$-Gorenstein and let $\!\cJ_{Y,m}$ be the unique ideal sheaf such that the image of $(\Omega_Y^{\dim Y})^{\otimes m}\to\omega_{Y,m}$ is given by $\!\!\cJ_{Y,m}\omega_{Y,m}$. Denef--Loeser \cite{DenefLoeser1999} introduced the Gorenstein measure 
\[
\mu_{Y}^{\Gor}(C) := \int_{C} \bL^{(1/m)\ord_{\cJ_{Y,m}}} \diff\mu_{Y}\in \widehat{\cM}_k[\bL^{1/m}]
\]
and proved that the Hodge--Deligne specialization of $\mu_{Y}^{\Gor}(\sL(Y))$ recovers the stringy Hodge numbers of $Y$. 
Recall the notion of generic $G$-representations.

\begin{definition}[{\cite[Definition 1.6]{SpenkoVanDenBergh17}}]
\label{def:generic-G-rep}
Given a finite-dimensional representation $V$ of a linear algebraic group $G$, let $U$ be the locus of points with closed orbit and trivial stabilizer. We say $V$ is \emph{generic} if $\codim(V\setminus U)\geq2$.
\end{definition}

We may now state our McKay correspondence:

\begin{theorem}[{McKay correspondence for linearly reductive groups}]
\label{thm:reductive-McKay}
If $V$ is a 
generic representation of a linearly reductive group $G$ such that $(\det V)^{\otimes n}$ is the trivial representation for some $n>0$, then $V/G$ is log-terminal and 
\[
	 \mu_{V/G}^{\Gor}(\sL(V/G)) 
 \ =\  \int_{\cC_{[V/G]}} \bL^{(1/m)\ord_{\!\cJ_{V\!/G,m}}\!\circ \sL(\underline{p}) - \het_{\cW([V/G])/(V/G)}} \diff\mu_{\cW([V/G]), \dim (V/G)}
\]
where $p\colon [V/G]\to V/G$ denotes the good moduli space map and where $m$ is such that $V/G$ is $m$-Gorenstein.
\end{theorem}

\begin{remark}
The classical McKay correspondence is for finite subgroups of $\SL(V)$. When $G$ is a linearly reductive subgroup scheme of $\SL(V)$, the $G$-representation $\det V$ is trivial (so we may take $n=1$); when $G$ is a \emph{finite} subgroup of $\SL(V)$, the genericity assumption on $V$ is also automatic. 
\end{remark}

\begin{remark}
Theorem \ref{thm:reductive-McKay} shows that the stringy Hodge numbers of $V/G$ are the coefficients of the Hodge-Deligne specialization of a canonical integral over warped arcs of $[V/G]$.
\end{remark}


\begin{acknowledgements}
This work benefited greatly from conversations with many individuals. It is our pleasure to thank Dan Abramovich, Kohei Aoyama, Anton Geraschenko, Changho Han, Jesse Kass, Dhruv Ranganathan, Karl Schwede, Wanchun Shen, Karen Smith, Jerry Wang, and Jonathan Wise. We thank Dan Abramovich for sharing with us toric examples that served as initial inspiration for the notion of warped arcs. We thank many colleagues who helped us with naming conventions, especially Dhruv Ranganathan for suggesting that we use the term warp. We are grateful to Dori Bejleri and Dan Edidin for their insights which led to Proposition \ref{prop:rep-open-gms}. We are indebted to Jack Hall for his patient explanations of many technical details that played a key role in Theorem \ref{thm:main--A0-algebraic}(\ref{thm:A0-algebraic}).
\end{acknowledgements}

\section{Warped arcs:~preliminaries}

Throughout this section, 
we assume the set-up in Notation \ref{not:main-gms} and let $\cU:=p^{-1}(U)$. Fix a field extension $k'/k$ and let $D=\Spec k'[[t]]$. 
We begin by showing that the canonical arcs introduced in Definition \ref{def:can-lift-toArtin-arc} are indeed warped arcs.


\begin{proposition}\label{prop:can-lift-toArtin-arc}
Let $\varphi\in(\sL(Y)\setminus\sL(Y\setminus U))(k')$. Then $\varphi^{\can}\in\cW(\cX)(D)$ and has integrity.
\end{proposition}

\begin{remark}\label{rmk:can-lift-toArtin-arc}
Upon showing that $\cW(\cX)$ is algebraic in Theorem \ref{thm:main--A0-algebraic}(\ref{thm:A0-algebraic}), Artin's criterion will tell us $\sL(\cW(\cX))(k')$ is naturally identified with $\cW(\cX)(D)$. Proposition \ref{prop:can-lift-toArtin-arc} can then be rephrased as saying $\varphi^{\can}\in\cC_\cX(k')$.
\end{remark}

\begin{proof}[{Proof of Proposition \ref{prop:can-lift-toArtin-arc}}]
Keep the notation of equations \eqref{eqn:phican} and \eqref{eqn:phicaneeqn}. We first show $f$ is representable with integrity and $\pi$ is finitely presented with affine diagonal. We see $p'$ is finitely presented and has affine diagonal since $p$ has these properties; note that $p$ is quasi-compact since this is part of the definition of being a good moduli space map. As good moduli space maps commute with base change, we see $p'$ is a good moduli space map. Since $\cX':=\cX\times_Y D$ is quasi-compact, applying Tags 0335 and 035S of \cite{stacks-project} to a finite type smooth cover of $\cX'$ shows that $q$ is finite; this implies $f$ has integrity and also that $\pi$ is finitely presented. The map $f$ is representable since both $\psi$ and $q$ are. Furthermore, $\Delta_{\cD/\cX'}$ is a closed immersion, so $\Delta_{\cD/D}$ is affine since it is a composition of $\Delta_{\cD/\cX'}$ and a pullback of $\Delta_{\cX'/D}$.

It remains to prove $\pi$ is a flat good moduli space map. By \cite[Lemma 4.14]{Alper}, we have a commutative diagram
\[
\xymatrix{
\cD\ar[r]^-{q}\ar[d]_-{\pi'}\ar[dr]^-{\pi} & \cX'\ar[d]^-{p'}\\
D'\ar[r]^-{g} & D
}
\]
where $\pi'$ is a good moduli space map and $D':=\Spec_D (\pi_*\cO_\cD) \to D$. Since $q$ is finite and $p'$ is a good moduli space map, Theorem 4.16(x) of (loc.~cit) shows that $\pi_*\cO_\cD$ is coherent, and so $g$ is finite. Note that all maps in the above diagram are isomorphisms over the generic point $D$. Consider the cartesian diagram
\[
\xymatrix{
\cD^\circ\ar[r]^-{i}\ar[d]_-{\pi^\circ}^-{\simeq} & \cD\ar[d]^-{\pi}\\
D^\circ\ar[r]^-{j} & D
}
\]
where $j\colon D^\circ\to D$ is the generic point of $D$. By construction, $\cD$ is integral, so $\cO_\cD$ is torsion-free. In particular, it has no $t$-torsion (where recall that $D=\Spec k'[[t]]$) and so $\pi$ is flat and the adjunction map $\cO_\cD\to i_*\cO_{\cD^\circ}$ is an injection. Since $\pi$ is cohomologically affine (as $q$ and $p'$ are), the induced map $\pi_*\cO_\cD\to \pi_*i_*\cO_{\cD^\circ}\simeq j_*\cO_{D^\circ}$ is also an injection; it follows that $\pi_*\cO_\cD$ is torsion-free, hence free. Since $g$ is generically an isomorphism, the map $\cO_D\to\pi_*\cO_\cD=\cO_{D'}$ is an isomorphism, i.e., $g$ is an isomorphism.
\end{proof}

The following lemma will be useful throughout this paper. Note that we are not assuming $(\pi,f)$ is a canonical arc.

\begin{lemma}\label{l:ArtinarcsinCX-are-isos-overD0}
Let $(\pi\colon\cD\to D,f\colon\cD\to\cX)\in\cW(\cX)(D)$. Letting $D^\circ\hookrightarrow D$ be the generic point and $\cD^\circ=\cD\times_D D^\circ$, if $f(\cD^\circ)\subset\cU$, then $\pi|_{D^\circ}\colon\cD^\circ\xrightarrow{\simeq}D^\circ$ is an isomorphism.
\end{lemma}
\begin{proof}
Since $f$ is representable, its base change $f\times_Y U\colon\cD^\circ\to\cU$ is also representable. As $\cU\simeq U$ is an algebraic space, $\cD^\circ$ is too. Lastly, since $\pi$ is a good moduli space map, $\pi\times_D D^\circ\colon\cD^\circ\to D^\circ$ is too, hence an isomorphism.
\end{proof}

We now prove that all twisted arcs factor through $\varphi^{\can}$. In fact, we show $\varphi^{\can}$ is universal for a much broader class of warped arcs as well.

\begin{proposition}\label{prop:twistede-discs-factor-through-can-lift-toArtin-arc}
Let $\varphi^{\can}$ be as in equation \eqref{eqn:phicaneeqn}. Let $\pi'\colon\cD'\to D$ be a warped disc with $\cD'$ normal and suppose we have a commutative 
\[
\xymatrix{
\cD'\ar@{-->}[r]^-{g}\ar[dr]_-{\pi'}\ar@/^1.3pc/[rr]^-{f'} & \cD\ar[d]^-{\pi}\ar[r]^-{f} & \cX\ar[d]^-{p}\\
& D\ar[r]^-{\varphi} & Y
}
\]
where $f'$ is representable, induces the map $\varphi$ on good moduli spaces, and maps the generic fiber of $\cD'$ to $\cU$, e.g., $f'$ can be a twisted arc. Then, up to unique isomorphism, there is a unique map $g$ making the diagram commute.
\end{proposition}
\begin{proof}
Keep the notation of equation \eqref{eqn:phican}. We obtain a commutative diagram
\[
\xymatrix{
\cD'\ar@{-->}[r]^-{g}\ar[dr]_-{\pi'}\ar@/^1.3pc/[rr]^-{q'} & \cD\ar[d]^-{\pi}\ar[r]^-{q} & \cX\times_Y D\ar[dl]^-{p'}\\
& D & 
}
\]
where $q$ is the normalization of the unique irreducible component $\cZ$ containing the generic fiber of $p'$; we must show the existence of $g$. Since $p'$ is isomorphism over the generic point of $D$ and $\pi'$ is as well by Lemma \ref{l:ArtinarcsinCX-are-isos-overD0}, we see $q'(\cD')$ is irreducible and contains the generic fiber of $p'$; it follows that $q'$ factors through $\cZ$ and is dominant onto $\cZ$. Note that $q'$ is representable (by algebraic spaces) as $f'$ is. Since $\cD'$ is normal and $q$ is the normalization of $\cZ$, the universal property of normalization yields a unique map $g$ up to unique isomorphism, see e.g., \cite[Tag 035Q]{stacks-project}; specifically, one applies (loc.~cit) to any smooth cover $V\to\cX\times_Y D$ by a scheme and to any \'etale cover of the algebraic space $V\times_{\cX\times_Y D} \cD'$ by a scheme.
\end{proof}

Next, we show that every twisted disc has integrity. This was the main inspiration for Definition \ref{def:integrity}.

\begin{proposition}\label{prop:twisted arcs-integrity}
Consider a commutative diagram
\[
\xymatrix{
\cD\ar[r]^f\ar[d]_-{\pi} & \cX\ar[d]^-{p}\\
D\ar[r]^-{\psi} & Y
}
\]
where $(\pi,f)$ is a twisted disc and $p$ is separated, e.g., if $\cX$ has finite inertia. Then $f$ has integrity.
\end{proposition}
\begin{proof}
Let $\cX':=\cX\times_Y D$, and $p'\colon\cX'\to D$ and $\psi'\colon\cX'\to\cX$ be the associated projections. We obtain an induced map $q\colon\cD\to\cX'$ such that $f=\psi'q$ and $\pi=p'q$. Since $f$ is representable, $q$ is as well. Since $p'$ is separated and, by definition, $\pi$ is proper and quasi-finite, we see $q$ is as well. Thus, $q$ is a representable, proper, quasi-finite morphism, hence finite.
\end{proof}

Lastly, we end the introduction with two examples.

\begin{example}\label{ex:Gmwt1-1}
Let $\cX=[\bA^2/\bG_m]$ where $\bG_m$ acts with weights $(1,-1)$. If $\bA^2$ has coordinates $x,y$, then $\cX\to \bA^1=\Spec k[xy]$ is a good moduli space map. Let $\varphi\colon D=\Spec k'[[t]]\to \bA^1$ be the map sending $xy$ to $t$. Then one computes that $\varphi^{\can}$ is given by the map
\[
f\colon\cD=[(\Spec k'[[t]][x,y]/(xy-t))/\bG_m]\to \cX
\]
where $\bG_m$ acts on $x,y$ with weights $1,-1$.
\end{example}

\begin{example}\label{ex:SLr}
The following example played a key role in \cite{SatrianoUsatine3}, where we constructed crepant resolutions of all log-terminal singularities using stacks. Similar to Example \ref{ex:Gmwt1-1}, suppose $\cX=[\bA^{r^2}/\SL_r]$ where $\bA^{r^2}$ is viewed as the space of $r\times r$ matrices and $\SL_r$ acts by left multiplication. Then the map $\cX\to \bA^1$ sending a matrix to its determinant is a good moduli space map. If $\varphi\colon D=\Spec k'[[t]]\to \bA^1$ is the map sending the coordinate on $\bA^1$ to $t$, then $\varphi^{\can}$ is given by the map
\[
f\colon\cD=[(\Spec k'[[t]][x_{ij}]/(\det(x_{ij})-t))/\SL_r]\to \cX.
\]
where $\SL_r$ acts on the matrix of indeterminates $(x_{i,j})$ by left multiplication.
\end{example}


\section{Deformation theory of warps}\label{sec:def-thy-Artin-arcs}

In this section, we prove that flat deformations of warps remain warps. This will be used throughout the proof of Theorem \ref{thm:main--A0-algebraic}(\ref{thm:A0-algebraic}). 

\begin{proposition}\label{prop:def-Artin-disc}
Let
\[
\xymatrix{
\cT\ar@{^{(}->}[r]^-{\iota}\ar[d]_-{\pi} & \cT'\ar[d]^-{\pi'}\\
\Spec A\ar@{^{(}->}[r]^-{i} & \Spec A'
}
\]
be a cartesian diagram with $A$ and $A'$ Artin local $k$-algebras, $i$ a nilpotent thickening, and $\pi'$ is flat. Then $\pi$ is a warp of $\Spec A$ if and only if $\pi'$ is a warp of $\Spec A'$.
\end{proposition}


%

We begin with a preliminary result.

\begin{lemma}\label{l:def-gms->gms}
Let
\[
\xymatrix{
\cY\ar@{^{(}->}[r]^-{\iota}\ar[d]_-{\pi} & \cY'\ar[d]^-{\pi'}\\
\Spec A\ar@{^{(}->}[r]^-{i} & \Spec A'
}
\]
be a cartesian diagram with $A$ and $A'$ Artin local $k$-algebras, $i$ a nilpotent thickening, and $\pi'$ flat. Assume $\cY'$ 
has affine diagonal.
\begin{enumerate}
\item\label{l:def-gms->gms::coh-affine} If $\pi$ is cohomologically affine, then $\pi'$ is as well.
\item\label{l:def-gms->gms::Stein} If $\pi$ is good moduli space morphism, then $\pi'$ is as well.
\end{enumerate}
\end{lemma}
\begin{proof}
Let $T=\Spec A$, $T'=\Spec A'$, and $T_0=\Spec A_0$, where $A_0$ is the residue field of $A$. Let $\cI=\ker(A'\to A)$. If $\fm$ is the maximal ideal of $A'$, then filtering $\cI$ by $\fm^\ell\cI$, we may assume $i$ is a small thickening, i.e., $\fm\cI=0$. Consider the cartesian diagram
\[
\xymatrix{
\cY_0\ar@{^{(}->}[r]^-{\iota_0}\ar[d]_-{\pi_0} & \cY\ar@{^{(}->}[r]^-{\iota}\ar[d]_-{\pi} & \cY'\ar[d]^-{\pi'}\\
T_0\ar@{^{(}->}[r]^-{i_0} & T\ar@{^{(}->}[r]^-{i} & T'
}
\]

We begin with (\ref{l:def-gms->gms::coh-affine}). Since $\pi$ is cohomologically affine and $\iota$ induces an isomorphism $\iota_{\red}\colon\cY_{\red}\xrightarrow{\simeq}\cY'_{\red}$, we see $\pi'_\red\colon\cY'_{\red}\to (\Spec A')_{\red}=\Spec A_0$ is cohomologically affine. Then \cite[Proposition 3.10(iii)]{Alper} shows $\pi'_{\red}$ is cohomologically affine. Note that $\cY$ is Noetherian since, by definition, $\pi$ is quasi-compact; it follows that $\cY'$ is Noetherian. Since $\cY'$ has affine diagonal, another application of (loc.~cit) shows that $\pi'$ is cohomologically affine.

We now turn to (\ref{l:def-gms->gms::Stein}). By (\ref{l:def-gms->gms::coh-affine}), we need only prove $\pi'$ is Stein. Since $\pi$ is a good moduli space map, $\pi_0$ is as well. Since $\fm\cI=0$, we may view $\cI$ as an ideal sheaf on $T_0$. We then have a short exact sequence
\[
0\to (i_0\circ i)_*\cI\to \cO_{T'}\to\cO_{T}\to 0
\]
Flatness of $\pi'$ tells us 
\[
0\to(\iota_0\circ\iota)_*(\pi')^*\cI\to \cO_{\cY'}\to\cO_{\cY}\to 0
\]
is exact. Then cohomological affineness yields a diagram
\[
\xymatrix{
0\ar[r] & \cI\otimes(\pi_0)_*\cO_{\cY_0}\ar[r]\ar[d]& \pi'_*\cO_{\cY'}\ar[r]\ar[d]&\pi_*\cO_{\cY}\ar[r]\ar[d] & 0\\
0\ar[r] & \cI\ar[r] & \cO_{T'}\ar[r] & \cO_{T}\ar[r] & 0
}
\]
where the rows are exact. The outer two vertical maps are isomorphisms since $\pi$ and $\pi_0$ are Stein. It follows that the middle vertical map is also an isomorphism, so $\pi'$ is Stein.
\end{proof}

\begin{proof}[{Proof of Proposition \ref{prop:def-Artin-disc}}]
If $\pi'$ is a warp, then $\pi$ is since the property of being a warp is stable under base change.

Assume now that $\pi$ is a warp. Since $\cT$ is quasi-compact, which is a topological notion, we see $\cT'$ is also quasi-compact. Then applying \cite[Tag 06AG]{stacks-project} to a smooth cover of $\cT'$, we see $\pi'$ is locally finitely presented, hence of finite presentation. Next, since the diagonal $\Delta_{\pi'}$ of $\pi'$ is representable and a flat deformation of the diagonal $\Delta_\pi$, we again see from (loc.~cit) that $\Delta_{\pi'}$ is affine.

We now show $\pi'$ is a good moduli space map. Since $\pi'$ is of finite presentation, we see $\cT'$ is locally Noetherian. As $\Delta_{\pi'}$ is affine and $\Spec A'$ is separated, hence has affine diagonal, we see $\cT'$ has affine diagonal. Lemma \ref{l:def-gms->gms} then shows $\pi'$ is a good moduli space morphism. 
\end{proof}

\section{Moduli of warped maps:~proof of Theorem \ref{thm:main--A0-algebraic}(\ref{thm:A0-algebraic})}

Throughout this section, we assume the 
set-up in Notation \ref{not:main}. 
Our main goal of this section is to prove that $\cW(\cX)$ is an Artin stack.
We do so by applying \cite[Main Theorem]{HallRydh-Artin-axioms}, which is a refined version of Artin's representability criterion. We verify the hypotheses of (loc.~cit) in each of the subsequent subsections.

\begin{remark}\label{rmk:LA0Xk-vs-A0XD}
Upon proving Theorem \ref{thm:main--A0-algebraic}(\ref{thm:A0-algebraic}), we will know $\sL(\cW(\cX))$ is a stack and every $\sL_n(\cW(\cX))$ is an algebraic stack. Furthermore, by Artin's criteria for algebraicity (\emph{cf}.~Proposition \ref{l:A0-effectivity}), we have a natural equivalence
\[
\cW(\cX)(k'[[t]])\xrightarrow{\simeq}\sL(\cW(\cX))(k')
\]
for all field extensions $k'/k$.
\end{remark}

\subsection{$\cW(\cX)$ is a stack that is limit-preserving}

The following propositions show, respectively, that $\cW(\cX)$ is limit-preserving and that it is a stack.

\begin{proposition}\label{l:A0-loc-fin-pres}
Let $\{\Spec A_j\}_{j\in J}$ be an inverse system of affine $k$-schemes with $\Spec A=\lim_j\Spec A_j$. Then the induced map
\[
\lim_\rightarrow\cW(\cX)(A_j)\to\cW(\cX)(A)
\]
is an equivalence of categories.
\end{proposition}
\begin{proof}
We first prove essential surjectivity. 
Let $(\pi\colon\cT\to\Spec A,f\colon\cT\to\cX)$ be an object of $\cW(\cX)(A)$. Let $f'\colon\cT\to\cX_A:=\cX\times_kA$ be the induced map. Since $\pi$ is of finite presentation and $\cX_A\to\Spec A$ is locally of finite presentation, \cite[Proposition B.1 and Proposition B.2(ii)]{Rydh2015} tell us there exists an index $j\in J$, an Artin stack $\pi_j\colon\cT_j\to\Spec A_j$ of finite presentation, a map $f'_j\colon\cT_j\to\cX_{A_j}:=\cX\times_kA_j$, and a cartesian diagram
\[
\xymatrix{
\cT_j\ar[d]^-{f'_j}\ar@/_1.3pc/[dd]_-{\pi_j} & \cT\ar[d]_-{f'}\ar@/^1.3pc/[dd]^-{\pi}\ar[l]\\
\cX_{A_j}\ar[d] & \cX_A\ar[d]\ar[l]\\
\Spec A_j & \Spec A\ar[l]
}
\]
By \cite[Proposition B.3]{Rydh2015}, after possibly enlarging the index $j$, we may assume $\pi_j$ is flat, $\Delta_{\pi_j}$ is affine, and $f'_j$ is representable, hence $f_j\colon\cT'_j\to\cX_{A_j}\to\cX$ is representable. Furthermore, after possibly enlarging $j$, \cite[Corollary 7.5]{AHRetalelocal} tells us $\pi_j$ is a good moduli space.

It remains to show full faithfulness. For this we may fix an index $0\in J$ and consider two objects $\gamma_0:=(\cT_0\to\Spec A_0,\cT_0\to\cX)$ and $\gamma'_0:=(\cT'_0\to\Spec A_0,\cT'_0\to\cX)$ of $\cW(\cX)(A_0)$; these induce representable maps $\cT_0\to\cX_{A_0}$ and $\cT'_0\to\cX_{A_0}$. A morphism $\gamma_0\to\gamma'_0$ is given by an equivalence class of pairs $(\phi,\alpha)$ with $\phi\colon\cT_0\times_{A_0} A\to\cT'_0\times_{A_0} A$ an isomorphism and $\alpha$ a $2$-isomorphism between the resulting maps $\cT_0\times_{A_0} A\to\cX_A$. Applying \cite[Propositions B.2(i) and B.3]{Rydh2015} with $X_0=\cT_0$, $Y_0=\cT'_0$, and $S_0=\cX_{A_0}$, this extends uniquely to a pair $(\phi_j,\alpha_j)$ with $\phi_j$ an isomorphism $\cT_j:=\cT_0\times_{A_0} A_j\to\cT'_0\times_{A_0} A_j=:\cT'_j$ and $\alpha_j$ a $2$-morphism between the maps $\cT_j\to\cX_{A_j}$. If $(\phi',\alpha')$ is in the same equivalence class as $(\phi,\alpha)$ and $\beta\colon\phi'\Rightarrow\phi$ with $\alpha'=\alpha\beta$, then again by (loc.~cit), for a sufficiently large index $j\in J$, $\beta$ extends to $\beta_j\colon\phi'_j\Rightarrow\phi_j$ and after possibly enlarging $j$, uniqueness tells us $\alpha'_j=\alpha_j\beta_j$. 
\end{proof}

\begin{proposition}\label{l:A0-stack}
$\cW(\cX)$ is a stack on the fppf site.
\end{proposition}
\begin{proof}
Let $\{T_i\to T\}_i$ be an fppf cover of schemes. We must show that the natural map 
\[
\cW(\cX)(T)\to \cW(\cX)(\{T_i\to T\}_i)
\]
to the category of descent data is an equivalence. We prove essential surjectivity; the proof of full faithfulness is easier and follows from analogous arguments. Let $T_{ij}:=T_i\times_T T_j$, $\gamma_i:=(\cT_i\to T_i,f_i\colon\cT_i\to\cX)\in\cW(\cX)(T_i)$, and $\alpha_{ij}\colon\gamma_i\times_{T_i}T_{ij}\xrightarrow{\simeq}\gamma_j\times_{T_j}T_{ij}$ be isomorphisms which satisfy the cocycle condition. This induces isomorphisms $\beta_{ij}\colon\cT_i\times_{T_i}T_{ij}\xrightarrow{\simeq}\cT_j\times_{T_j}T_{ij}$. In order to descend the $\cT_i$ to an Artin stack $\cT$ over $T$, one needs our given $2$-isomorphisms $\beta_{jk}\circ\beta_{ij}\Rightarrow\beta_{ik}$ on the triple fiber products $T_{ijk}:=T_i\times_T T_j\times_T T_k$ to satisfy a cocycle condition on the quadruple fiber products; in our case, this is automatic since the $\beta_{ij}$ commute with the representable maps to $\cX$. This therefore yields our desired Artin stack $\pi\colon\cT\to T$. Note that $\pi$ is a flat finitely presented good moduli space map with affine diagonal since these properties may be checked fppf locally, see \cite[Proposition 4.7]{Alper}.

Next, we must descend the maps $f_i\colon\cT_i\to\cX$ to $f\colon\cT\to\cX$. We have $2$-isomorphisms $\beta'_{ij}\colon f_j\circ\beta_{ij}\Rightarrow f_i$ on the $T_{ij}$, and again using representability of the maps to $\cX$, we see that we have $\beta'_{jk}\circ\beta'_{ij}=\beta'_{ik}$ on $T_{ijk}$. Thus, the $f_i$ descend to $f$, which is necessarily representable since this may be checked fppf locally on the base. 
\end{proof}

\subsection{Algebraization of formal moduli}
In this section, we prove that $\cW(\cX)$ satisfies the weak effectivity criterion in \cite{HallRydh-Artin-axioms}. We begin with a general result showing opennness of representable maps.

\begin{proposition}\label{prop:rep-open-gms}
Let $T$ be a scheme and let $f\colon\cY\to\cY'$ be a morphism of Artin stacks that are locally of finite type over $T$. Assume that $\Delta_{\cY/T}$ and $\Delta_{\cY'/T}$ are affine, and $\cY\to T$ is a good moduli space. Then there is an open subscheme $U\subset T$ such that a morphism $W\to T$ factors (uniquely) through $U$ if and only if the base change $f_W\colon\cY_W\to\cY'_W$ is representable.
\end{proposition}
\begin{proof}
Observe that the hypotheses on the diagonals of $\Delta_{\cY/T}$ and $\Delta_{\cY'/T}$ imply $\Delta_f$ is affine. Applying \cite[Tag 05X8]{stacks-project} to $\Delta_f$, we may assume it is unramified, i.e., $f$ is relatively Deligne--Mumford; note this is a necessary condition for representability. Since unramified maps are locally quasi-finite by Tag 02V5 of (loc.~cit), and since $\Delta_f$ is quasi-compact, we may assume $\Delta_f$ is both quasi-finite and affine.

Let $V'\to\cY'$ be a smooth cover by a scheme, which is an affine map since $\Delta_{\cY'/T}$ is affine. Let $\cV:=\cY\times_{\cY'}\cV'$ and $\rho\colon\cV\to\cY$ be the induced smooth cover. Note that if $W\to T$ is a map, then $f_W$ is representable if and only if $\cV\times_T W$ is an algebraic space. Furthermore, since $\Delta_f$ is quasi-compact and separated, Corollary 2.2.7 of \cite{ConradKM} shows this is equivalent to $\cV\times_T K$ being an algebraic space for every geometric point $\Spec K\to W$. Therefore, we must prove that the set of points $t\in T$ with $\cV_t:=\cV\times_T k(t)$ an algebraic space forms an open subset of $T$.

We see $\cV$ is Deligne--Mumford as $f$ is relatively Deligne--Mumford. Since the structure map $\pi\colon\cY\to T$ is a good moduli space, we see $p\colon\cV\to\Spec_T(\rho_*\pi_*\cO_\cV)=:V$ is a good moduli space by \cite[Lemma 4.14]{Alper}. As good moduli spaces are adequate moduli spaces \cite[p.~490]{Alperadequate}, and since $\cV$ has separated quasi-finite diagonal, Theorem 8.3.2 of (loc.~cit) shows that $\cV$ admits a separated coarse space map. By universality of coarse spaces and good moduli spaces, we see $p\colon\cV\to V$ itself must be a separated coarse space map. In particular, $\cV$ is separated over $T$.

Theorems 2.2.5(2) and Corollary 2.2.7 of \cite{ConradKM} then show that there is a maximal open substack $\cV^\circ\subset\cV$ which is an algebraic space and it is characterized by the property that its geometric points have trivial automorphism groups. Since $\rho$ is flat, $\cY^\circ:=\rho(\cV^\circ)$ is open in $\cY$, and we see for any $t\in T$, the map $f_t$ is representable if and only if $\cY^\circ_t=\cY_t$; moreover, the formation of $\cY^\circ$ commutes with arbitrary base change on $T$ since it is characterized by a property on geometric points. We then see that $T\setminus\pi(\cY\setminus\cY^\circ)$ is open (as $\pi$ is universally closed), and consists precisely of those $t\in T$ for which $f_t$ is representable.
\end{proof}


\begin{proposition}\label{l:A0-effectivity}
Let $B$ be a complete local Noetherian $k$-algebra with maximal ideal $\fm$ and $B/\fm$ locally of finite type over $k$. Then
\[
\cW(\cX)(B)\to\lim_n\cW(\cX)(B/\fm^{n+1})
\]
is an equivalence of categories.
\end{proposition}
\begin{proof}
Let $B_n:=B/\fm^{n+1}$. We begin by proving essential surjectivity. Let $(\gamma_n)_n$ be an object of $\lim_n\cW(\cX)(B/\fm^n)$ where $\gamma_n:=(\pi_n\colon\cT_n\to \Spec B_n,f_n\colon\cT_n\to\cX)\in\cW(\cX)(B_n)$. Since $\pi_0$ has affine diagonal, by \cite[Corollary 4.14(1)]{AHRLuna}, we see $\cT_0$ is linearly fundamental in the sense of \cite[Definition 2.7]{AHRetalelocal}. Then Theorem 1.10 of (loc.~cit) shows that the completion $\cT$ of the adic sequence $\{\cT_n\}_n$ exists, see Definition 1.9 of (loc.~cit); furthermore, the proof shows that $\cT_n$ is the $n$-th infinitesimal thickening of $\cT_0$ in $\cT$. 
We then have maps $f\colon\cT\to\cX$ and $\pi\colon\cT\to\Spec B$ by \S1.7.6 of (loc.~cit). 

We prove that $\gamma:=(\pi,f)\in\cW(\cX)(B)$ and that $\gamma\otimes_B B_n=\gamma_n$. These statements follow upon showing $\cT_n\to\cT\times_B B_n$ is an isomorphism. Indeed, the local criterion for flatness applied to a smooth cover of $\cT$ tells us $\pi$ is flat. 
Since $\cT$ has affine diagonal, $\pi$ does as well. By construction, $\cT$ is Noetherian and $(\cT_0,\cT)$ is a local pair. Example 3.9 of (loc.~cit) shows that $(\Spec B_0,\Spec B)$ is a Noetherian complete pair. Since $\cT$ is linearly fundamental by Theorem 1.10 of (loc.~cit), $\cT$ has an affine good moduli space. It then follows from Lemma 7.11 of (loc.~cit) that $\pi$ is a good moduli space of finite presentation. 
Lastly, since $\cX_B:=\cX\times_k B\to \Spec B$ has affine diagonal, and the induced map $g\colon\cT\to\cX_B$ is representable over the closed point of $\Spec B$, Proposition \ref{prop:rep-open-gms} shows $g$ is representable; in particular, $f$ is representable.

We turn now to the claim that $\cT_n\to\cT\times_B B_n$ is an isomorphism. For this, fix $n$, and let $\cZ:=\cT_n$ and $\cY:=\cT\times_B B_n$; note that both of $\cZ$ and $\cY$ are naturally closed substacks of $\cT$, and so it suffices to construct a section of our closed immersion $\cZ\to\cY$. For all $m$, the induced map from $\cZ_m:=\cZ\times_\cT \cT_m$ to $\cY_m:=\cY\times_\cT \cT_m$ is an isomorphism as both $\cZ_m$ and $\cY_m$ are identified with the closed substack $\cT_{\min(n,m)}$ of $\cT_m$; note this makes use of the fact that $\cT_n$ is $n$-th infinitesimal neighbourhood of $\cT_0$ in $\cT$. By Lemma 3.5(1) of (loc.~cit), we see $\cY$ is coherently complete along $\cY_0$. Furthermore, we see $\cY_m$ is the $m$-th infinitesimal neighbourhood of $\cY_0$ in $\cY$ since $\cY_0\to\cY_m\to\cY$ is the pullback of $\cT_0\to\cT_m\to\cT$. Thus, $\cY$ is the completion of the adic sequence $\{\cY_m\}$, hence our maps $\cY_m\to\cZ_m\to\cZ$ yield a section $\cY\to\cZ$ by Tannaka duality, see \S1.7.6 of (loc.~cit). 

We now prove full faithfulness. Let
\[
\gamma^{(i)}:=(\pi^{(i)}\colon\cT^{(i)}\to \Spec B,f^{(i)}\colon\cT^{(i)}\to\cX)\in \cW(\cX)(B)
\]
for $i=1,2$ and let 
\[
\gamma^{(i)}_n:=(\pi^{(i)}_n\colon\cT^{(i)}_n\to \Spec B,f^{(i)}_n\colon\cT^{(i)}_n\to\cX)\in \cW(\cX)(B_n)
\]
denote the objects induced by pullback. Additionally suppose we are given a compatible sequence of isomorphisms $\phi_n\colon\cT^{(1)}_n\to\cT^{(2)}_n$ over $B_n$ and $2$-isomorphisms $\alpha_n\colon f_n^{(2)}\phi_n\Rightarrow f_n^{(1)}$. Let $\iota^{(i)}_n\colon\cT^{(i)}_n\hookrightarrow\cT^{(i)}$ denote the closed immersion. Since $\cT$ is Noetherian with affine diagonal, \cite[Corollary 1.7]{AHRetalelocal} shows that $\cT^{(1)}$ is the completion of the adic sequence $\{\cT^{(1)}_n\}_n$. 
Then \S1.7.6 of (loc.~cit) shows there is a unique map $\phi\colon\cT^{(1)}\to\cT^{(2)}$ extending the family of maps $\iota^{(2)}_n\phi_n\colon\cT^{(1)}_n\to\cT^{(2)}$; it is an isomorphism since swapping the roles of $\cT^{(1)}$ and $\cT^{(2)}$ constructs its inverse. Since $\cT^{(i)}_n$ is the $n$-th infinitesimal neighborhood of $\cT^{(i)}_0$ in $\cT^{(i)}$, we see $\phi\iota^{(1)}_1$ factors through $\phi_n$. The uniqueness statement of \S1.7.6 of (loc.~cit) applied to the maps $\pi^{(1)}$ and $\phi\pi^{(2)}$ shows they are equal. It follows that $\phi_n$ is naturally isomorphic to $\phi\otimes_B B_n$. Another application of \S1.7.6 of (loc.~cit) shows that our $2$-isomorphisms $\alpha_n$ yield a unique $2$-isomorphism $\alpha\colon f^{(2)}\phi\Rightarrow f^{(1)}$.

This shows full faithfulness modulo one claim:~if $(\phi,\alpha)$ and $(\phi',\alpha')$ are morphisms $\gamma^{(1)}\to\gamma^{(2)}$ which yield morphisms $\gamma_n^{(1)}\to\gamma_n^{(2)}$ in the same equivalence class for all $n$, then $(\phi,\alpha)$ and $(\phi',\alpha')$ are also in the same equivalence class. Let $(\phi_n,\alpha_n)$ and $(\phi'_n,\alpha'_n)$ denote the induced maps. Then for each $n$, there exists $\beta_n\colon\phi_n\Rightarrow\phi'_n$ such that $\alpha_n=\alpha'_n\beta_n$. By uniqueness of the $\beta_n$, see Remark \ref{rmk:2-stack-1-stack}, we see $\beta_n$ is the pullback of $\beta_{n+1}$. Hence, by \cite[\S1.7.6]{AHRetalelocal}, there exists $\beta$ which pulls back to the $\beta_n$ and satisfies $\alpha=\alpha'\beta$ (by the uniqueness statement of (loc.~cit)). Thus,  $(\phi,\alpha)$ and $(\phi',\alpha')$ are in the same equivalence class.
\end{proof}

\subsection{$\textbf{Aff}$-homogeneity}

We refer the reader to \cite[\S1]{HallRydh-Artin-axioms} for the relevant notions of $P$-homogeneity. The following result establishes $\textbf{Aff}$-homogeneity for $\cW(\cX)$. This implies $\textbf{Art}^{\textbf{triv}}$-homogeneity and $\textbf{Nil}$-homogeneity which are used in the proof of Theorem \ref{thm:main--A0-algebraic}(\ref{thm:A0-algebraic}). Note that by \cite[Lemma 1.5(1)]{HallOpenVersal}, every cocartesian $\textbf{Aff}$-nil square is geometric and so it suffices to verify $\textbf{Aff}$-homogeneity for geometric pushouts.

\begin{proposition}\label{prop:Arttriv}
Consider a cocartesian square
\[
\xymatrix{
V\ar[r]^-{p}\ar@{^{(}->}[d]^-{i} & T\ar@{^{(}->}[d]^-{i'}\\
V'\ar[r]^-{p'} & T'
}
\]
where $i$ and $i'$ are locally nilpotent closed immersions, $p$ and $p'$ are affine, and the natural map
\[
\cO_{T'}\xrightarrow{\simeq} i'_*\cO_T\times_{p'_*i_*\cO_V} p'_*\cO_{V'}
\]
is an isomorphism. Then the induced map 
\[
\cW(\cX)(T')\to \cW(\cX)(T)\times_{\cW(\cX)(V)}\cW(\cX)(V')
\]
is an equivalence of categories.
\end{proposition}
\begin{proof}
We show essential surjectivity. Let 
\[
\gamma_Z:=(\pi_Z\colon\cT_Z\to Z,g_Z\colon\cT_Z\to\cX)\ \in\ \cW(\cX)(Z)
\]
for $Z\in\{V,V',T\}$, and fix isomorphisms $p^*\gamma_{T} \xleftarrow{\simeq}\gamma_V\xrightarrow{\simeq} i^*\gamma_{V'}$. We denote by $f_Z\colon\cT_Z\to\cX_Z:=\cX\times_k Z$ the map induced by $g_Z$. Theorem 4.2 of \cite{AHHLR} shows that the pushout $\cT_{T'}$ of the diagram $\cT_T\leftarrow \cT_V\rightarrow\cT_{V'}$ exists and is quasi-compact. Since $\cT_V$, $\cT_{V'}$, and $\cT_T$ have compatible maps to $\cX$ and $T'$, the universal property of $\cT_{T'}$ yields maps $\pi_{T'}\colon\cT_{T'}\to T'$ and $g_{T'}\colon\cT_{T'}\to\cX$; the latter then yields a map $f_{T'}\colon\cT_{T'}\to\cX_{T'}$. Let $\gamma_{T'}=(\pi_{T'},g_{T'})$.

By \cite[Lemma 4.4]{AHHLR}, $\gamma_{T'}$ base changes over $T$ (resp.~$V'$) to $\gamma_{T}$ (resp.~$\gamma_{V'}$). Furthermore, (loc.~cit) proves that $\pi_{T'}$ is flat and locally finitely presented since $\pi_T$ is; since $\cT_{T'}$ is quasi-compact, $\pi_{T'}$ is finitely presented. To verify $\pi_{T'}$ is a good moduli space map with affine diagonal and that $f_{T'}$ is representable, it suffices to look fppf locally on $T'$, where the map $i'$ is a nilpotent thickening, in which case $\pi_{T'}$ is a flat deformation of $\pi_T$. Hence, Proposition \ref{prop:def-Artin-disc} shows that $\pi_{T'}$ is a warp. Lastly, since $f_{T'}$ is a flat deformation of $f_T$, 
it is representable by \cite[Tag 0CJ9]{stacks-project}. Thus, $\gamma_{T'}\in\cW(\cX)(T')$.

We now turn to full faithfulness. Let $\gamma_{T'}^{(i)}:=(\pi^{(i)}_{T'}\colon\cT^{(i)}_{T'}\to T',f^{(i)}_{T'}\colon\cT^{(i)}_{T'}\to\cX)$ be objects of $\cW(\cX)(T')$ for $i=1,2$. These induce objects
\[
\gamma_Z:=(\pi^{(i)}_Z\colon\cT^{(i)}_Z\to Z,f^{(i)}_Z\colon\cT^{(i)}_Z\to\cX)\ \in\ \cW(\cX)(Z)
\]
for $Z\in\{V,V',T\}$. Since $\pi^{(i)}_{T'}$ is flat and the square in the statement of the proposition is a geometric pushout, \cite[Lemma 4.3(2)]{AHHLR} shows that
\[
\xymatrix{
\cT^{(i)}_V\ar@{^{(}->}[d]\ar[r] & \cT^{(i)}_{T}\ar@{^{(}->}[d]\\
\cT^{(i)}_{V'}\ar[r] & \cT^{(i)}_{T'}
}
\]
is also a geometric pushout.

Full faithfulness now follows from the universal property of pushouts. Indeed, we are given compatible maps $\alpha_Z\colon\cT^{(1)}_Z\to\cT^{(2)}_Z$, for $Z\in\{V,V',T\}$, and compatible choices of $2$-isomorphisms $f^{(2)}_Z\circ\alpha_Z\Rightarrow f^{(1)}_Z$. By the universal property of pushouts, we obtain a map $\alpha_{T'}\colon\cT^{(1)}_{T'}\to\cT^{(2)}_{T'}$ as well as maps $f^{(i)}_{T'}\colon\cT^{(i)}_{T'}\to\cX$; moreover, our given compatible $2$-isomorphisms $f^{(2)}_Z\circ\alpha_Z\Rightarrow f^{(1)}_Z$, for $Z\in\{V,V',T\}$, extend to a unique $2$-isomorphism $f^{(2)}_{T'}\circ\alpha_{T'}\Rightarrow f^{(1)}_{T'}$. For faithfulness, we are given $(\alpha_{T'},\iota_{T'})$ and $(\alpha'_{T'},\iota'_{T'})$ with $\iota_{T'}\colon f^{(2)}_{T'}\circ\alpha_{T'}\Rightarrow f^{(1)}_{T'}$ and $\iota'_{T'}\colon f^{(2)}_{T'}\circ\alpha'_{T'}\Rightarrow f^{(1)}_{T'}$. Let $(\alpha_Z,\iota_Z)$ and $(\alpha'_Z,\iota'_Z)$ be the induced isomorphisms over $Z\in\{V,V',T\}$, and we know $(\alpha_Z,\iota_Z)$ and $(\alpha'_Z,\iota'_Z)$ are in the same equivalence class. This means there is $\beta_Z\colon\alpha_Z\Rightarrow\alpha'_Z$ such that $\iota_Z=\iota'_Z\beta_Z$. By uniqueness of $\beta_V$ (see Remark \ref{rmk:2-stack-1-stack}), the $\beta_Z$ are compatible, and hence we obtain $\beta_{T'}\colon\alpha_{T'}\Rightarrow\alpha'_{T'}$ such that $\iota_{T'}=\iota'_{T'}\beta_{T'}$.
\end{proof}

\subsection{Boundedness of deformations and obstructions}

\begin{proposition}\label{prop:bounded-defs-obstr}
$\cW(\cX)$ has bounded and constructible automorphisms and deformations as well as constructible obstructions (in the sense of \cite{HallRydh-Artin-axioms}).
\end{proposition}
\begin{proof}
Let $\gamma:=(\pi\colon\cT\to T_0,f\colon\cT\to\cX)\in\cW(\cX)(T_0)$, where $T_0$ is an affine $k$-scheme which is locally of finite type. Let $T_0\hookrightarrow T'$ be a square-zero thickening with ideal sheaf $\cI$. Proposition \ref{prop:def-Artin-disc} shows that deformations of $\cT\to T_0$ over $T'$ as a warp are the same as deformations as an Artin stack, and hence, deformations of $\gamma$ are the same as deformations of the morphism $f$. Thus, automorphisms and deformations are classified by $\Ext^i(L_f,\pi^*\cI)$ for $i=0,1$ and the obstruction class lives in $\Ext^2(L_f,\pi^*\cI)$, see e.g., \cite[Theorem 1.4]{OlssonDefRep}. Since $\pi$ is a good moduli space morphism, $\pi_*\mathcal{E}xt^i(L_f,\pi^*\cI)$ is coherent and its global sections agree with $\Ext^i(L_f,\pi^*\cI)$. This proves boundedness of automorphisms, deformations, and obstructions.

For constructibility, assume $T_0$ is integral and consider the case where $\cI=\cO_{T_0}$. We make use of the spectral sequence 
\[
E_2^{p,q}=\Ext^p(\cH^{-q}(L_f),\cO_{\cT})\Rightarrow \Ext^{p+q}(L_f,\cO_{\cT}).
\]
As in the previous paragraph, $\Ext^p(\cH^{-q}(L_f),\cO_{\cT})$ is the global sections of the coherent sheaf $\pi_*\mathcal{E}xt^p(\cH^{-q}(L_f),\cO_{\cT})$. Fix an integer $i$. Then only finitely many $E_2^{p,q}$-terms are relevant to the computation of $\Ext^i(L_f,\cO_{\cT})$. Since $\cH^{-q}(L_f)$ are coherent sheaves, we may choose a non-empty open subset $U\subset T_0$ such that all relevant $E_2^{p,q}$-terms are locally free. Since the formation of $L_f$ and good moduli spaces commute with flat base change, the formation of $\pi_*\mathcal{E}xt^p(\cH^{-q}(L_f),\cO_{\cT})$ commutes with base change to $U$, i.e., we may replace $T_0$ by $U$. The spectral sequence then collapses with the only remaining terms being $\Hom(\cH^{-q}(L_f),\cO_{\cT})$, all of which commute with base change to the fibers of $U$. This proves constructibility of automorphisms and deformations, and by \cite[Lemma 7.5(1)]{HallRydh-Artin-axioms}, constructibility of obstructions; note that $\cW(\cX)$ is $\textbf{Nil}$-homogeneous since it is $\textbf{Aff}$-homogeneous by Proposition \ref{prop:Arttriv}.
\end{proof}

\subsection{Proof of Theorem \ref{thm:main--A0-algebraic}(\ref{thm:A0-algebraic})}


Propositions \ref{l:A0-stack}, \ref{l:A0-loc-fin-pres}, and \ref{l:A0-effectivity} verify hypotheses (1)--(3), respectively, of \cite[Main Theorem]{HallRydh-Artin-axioms}. Proposition \ref{prop:Arttriv} shows $\textbf{Aff}$-homogeneity of $\cW(\cX)$, which implies both $\textbf{Art}^{\textbf{triv}}$-homogeneity (i.e., condition (4) of (loc.~cit)) and $\textbf{Nil}$-homogeneity (an implicit hypothesis in conditions (5) and (6)). Proposition \ref{prop:bounded-defs-obstr} verifies conditions (5a), (5b), and (6b); conditions (5c) and (6c) are satisfied since $\Spec k$ is Jacobson, see condition ($\alpha$). It follows that $\cW(\cX)$ is a locally finitely presented Artin stack.\hfill $\square$

\section{Functoriality and the main component:~proof of Theorem \ref{thm:main--A0-algebraic}(\ref{prop:X->A0X-open immersion})}
\label{sec:functoriality}

We assume the set-up in Notation \ref{not:main}. 
Recall from the introduction that there is a natural map $\tau\colon\cX\to\cW(\cX)$ sending a map $f\colon T\to\cX$ to the trivial warped map $(\id:T\to T,f\colon T\to\cX)$. We begin this section by proving Theorem \ref{thm:main--A0-algebraic}(\ref{prop:X->A0X-open immersion}), i.e., $\tau$ is an open immersion; this shows that there is a ``main component'' of $\cW(\cX)$, given by the closure of $\tau(\cX)$. 



\begin{proof}[{Proof of Theorem \ref{thm:main--A0-algebraic}(\ref{prop:X->A0X-open immersion})}]
We prove $\tau$ is representable, flat, locally finitely presented, and a monomorphism. This implies $\tau$ is an open immersion by \cite[Tag 025G]{stacks-project}.

By Theorem \ref{thm:main--A0-algebraic}(\ref{thm:A0-algebraic}), $\cW(\cX)$ is a locally finitely presented Artin stack. Since $\cX$ is as well, the map $\tau$ is locally finitely presented. 

Next, let $f$ and $g$ be two morphisms $T\to\cX$, and suppose we have an isomorphism $\tau(f)\Rightarrow \tau(g)$. In other words, we are given an isomorphism $h\colon T\to T$ and a $2$-isomorphism $\alpha\colon fh\Rightarrow gh$ such that $\id_T\circ h=\id_T$. This implies $h=\id_T$ and $\alpha\colon f\Rightarrow g$ is an isomorphism, proving that $f$ is isomorphic to $g$. This shows $\tau$ is a monomorphism. Furthermore, if we only consider the case where $f=g$, then we have shown that the automorphisms of $\tau(f)$ are precisely automorphisms of the map $f\colon T\to\cX$. Thus, $\tau$ induces a bijection on automorphism groups, and is therefore representable.

Lastly, we must show $\tau$ is flat. Since $\tau$ is locally finitely presented, it is enough to prove $\tau$ is formally smooth. Let $T'$ be an affine scheme and $T\hookrightarrow T'$ a nilpotent thickening. Then we must show the existence of a dotted arrow
\[
\xymatrix{
T\ar[r]\ar@{^{(}->}[d] & \cX\ar[d]^-{\tau}\\
T'\ar[r]\ar@{-->}[ur] & \cW(\cX)
}
\]
in the above diagram. In other words, we have a morphism $f\colon T\to\cX$, a flat deformation $\pi'\colon\cT'\to T'$ of $\id_T\colon T\to T$, and a representable map $f'\colon\cT'\to\cX$ deforming $f$. Since the only deformations of $\id_T$ are again the identity map, we see that $\pi'$ is isomorphic to $\id_{T'}$, i.e., this constructs our desired dotted arrow.
\end{proof}

We next discuss functoriality properties of $\cW(\cX)$. First, if $h\colon\cX'\to\cX$ is a representable map, we obtain a natural map $\cW(h)\colon\cW(\cX')\to\cW(\cX)$ sending $(\cT\to T,\cT\to\cX')$ to $(\cT\to T,\cT\to\cX'\xrightarrow{h}\cX)$. 

Next, given a map $p\colon\cX\to Y$ to an algebraic space, we obtain an induced map
\[
\underline{p}\colon \cW(\cX)\to Y
\]
as we now describe. For any algebraic space $\gamma\colon T\to\cW(\cX)$ we must construct functorial maps $\overline{\gamma}\colon T\to Y$. 
Let $\gamma=(\pi\colon\cT\to T, f\colon \cT\to\cX)$. Since $\pi$ is a good moduli space, by 
\cite[Theorem 3.12]{AHRetalelocal}, 
we obtain a unique map $\overline{\gamma}$ making the diagram
\[
\xymatrix{
\cT\ar[r]^-{f}\ar[d]_-{\pi} & \cX\ar[d]^-{p}\\
T\ar[r]^-{\overline{\gamma}} & Y
}
\]
commute.

\begin{remark}\label{rmk:X-gms->havemapA0X->Y}
From $\underline{p}$, we obtain an induced map $\sL(\underline{p})\colon\sL(\cW(\cX))\to\sL(Y)$. If $U\subset Y$ is any open set over which $p$ is an isomorphism and $\cU=p^{-1}(U)$, then this yields a map
\[
\sL(\cW(\cX))\setminus\sL(\cW(\cX\setminus\cU))\longrightarrow\sL(Y)\setminus\sL(Y\setminus U).
\]
\end{remark}

\begin{proposition}\label{prop:X-gms->A0Xgms}
Let $p\colon\cX\to Y$ be a map to an algebraic space. Then
\begin{enumerate}
\item\label{commutesWithTau} $p=\underline{p}\circ\tau$.

\item\label{AoXbaseChangesWell} if $g$ is a locally finitely presented map of algebraic spaces and the diagram
\[
\xymatrix{
\cX'\ar[r]^-{h}\ar[d]_-{p'} & \cX\ar[d]^-{p}\\
Y'\ar[r]^-{g} & Y
}
\]
is cartesian, then 
\[
\xymatrix{
\cW(\cX')\ar[r]^{\cW(h)}\ar[d]_-{\underline{p}'} & \cW(\cX)\ar[d]^-{\underline{p}}\\
Y'\ar[r]^-{g} & Y
}
\]
is as well.
\end{enumerate}
\end{proposition}
\begin{proof}
Part (\ref{commutesWithTau}) follows immediately from the definitions.

For part (\ref{AoXbaseChangesWell}), first note that $\cX'$ is locally Noetherian and so, by Theorem \ref{thm:main--A0-algebraic}(\ref{thm:A0-algebraic}), $\cW(\cX')$ is a locally finitely presented algebraic stack. One readily checks that $\underline{p}\circ \cW(h)=g\circ\underline{p}'$. It remains to check $\cW(\cX')=\cW(\cX)\times_Y Y'$. For this, let $\gamma\colon T\to\cW(\cX)$ and $q\colon T\to Y'$ such that $\underline{p}\circ\gamma=g\circ q$. Letting $\gamma=(\pi\colon\cT\to T,f\colon\cT\to\cX)$, we have a commutative diagram
\[
\xymatrix{
\cT\ar[rr]^-{f}\ar[d]_-{\pi} & & \cX\ar[d]^-{p}\\
T\ar[r]^-{q} & Y'\ar[r]^-{g} & Y
}
\]
which induces a representable map $f'\colon\cT\to\cX'$, and hence a warped map $(\pi,f')\in\cW(\cX')(T)$. We have now checked essential surjectivity of $\cW(\cX')\to\cW(\cX)\times_Y Y'$ on fibers. 

For full faithfulness, suppose we have warped maps $(\pi\colon\cT\to T,f\colon\cT\to\cX')$ and $(\pi'\colon\cT'\to T,f'\colon\cT'\to\cX')$. An isomorphism in $(\cW(\cX)\times_Y Y')(T)$ is given by an equivalence class $(\phi,\alpha)$ with $\phi\colon\cT'\xrightarrow{\simeq}\cT$ and $\alpha\colon hf'\Rightarrow hf\phi$ such that $\pi'=\pi\phi$. Since $\cX'=\cX\times_T T'$, we know there is a unique $\beta\colon f'\Rightarrow f\phi$ extending $\alpha$. This proves fullness. For faithfulness, one notes that if $(\phi,\beta)$ and $(\phi',\beta')$ are isomorphisms in $\cW(\cX')(T)$ which induce the same equivalence class in $\cW(\cX)(T)$, then by definition there exists $\eta\colon\phi\Rightarrow\phi'$ such that $h(\beta')=hf(\eta)\circ h(\beta)$. Since $\cX'=\cX\times_T T'$, this therefore implies that $\beta'=f(\eta)\circ \beta$, proving that $(\phi,\beta)$ and $(\phi',\beta')$ define the same equivalence class.
\end{proof}

\section{Quasi-affineness of inertia:~proof of Theorem \ref{thm:main--A0-algebraic}(\ref{thm:gms->A0XsepDiag-aff-geom-stab})}

Throughout this section, we assume Notation \ref{not:main} and prove the following strong version of Theorem \ref{thm:main--A0-algebraic}(\ref{thm:gms->A0XsepDiag-aff-geom-stab}).


\begin{theorem}\label{thm:sep-diag-coh-affine}
If $\cX$ admits a scheme as a good moduli space, then $I_{\cW(\cX)}\to\cW(\cX)$ is quasi-affine. 
In particular, $\cW(\cX)$ has separated diagonal and affine geometric stabilizers.
\end{theorem}


Our first goal is to show that if $\cT$ and $\cT'$ are warps of $T$, then the sublocus of the Hom-stack $\uHom_T(\cT,\cT')$ (see \cite[Theorem 5.10]{AHRLuna}) parameterizing isomorphisms is open. This requires several preliminary results, some of which will be useful in our proof of Theorem \ref{thm:main-canonical-cylinder}, specifically when we characterize the canonical cylinder $\cC'_\cX$; see Definition \ref{def:can-cylinder} and Theorem \ref{thm:main-can-cylinder}.

\begin{proposition}\label{prop:finiteness-property--->finite}
Let $S$ be a scheme and let
\[
\xymatrix{
\cY\ar[r]^-{f}\ar[dr]_-{\pi} & \cY'\ar[d]^-{\pi'}\\
 & Y
}
\]
be a commutative diagram of locally Noetherian 
algebraic $S$-stacks, where $\pi$ and $\pi'$ are good moduli space maps. Suppose $\Delta_{Y/S}$ is quasi-affine and $\Delta_{\cY'/Y}$ is affine, e.g.~this holds if $\Delta_{Y/S}$ is quasi-compact and $\Delta_{\cY'/S}$ is affine. 
Assume either:
\begin{enumerate}
\item $f$ is surjective, or
\item\label{prop:finiteness-property--->finite::main} $Y$ is proper over $S$ and for all $s\in S$, the map on fibers $f_s\colon\cY_s\to\cY'_s$ sends closed points to closed points.
\end{enumerate}
If $f$ is representable and quasi-finite, then $f$ is finite.
\end{proposition}
\begin{proof}
We first show $f$ is affine, and in particular, separated. By hypothesis, $\Delta_{Y/S}$ is quasi-affine and $\Delta_{\cY'/Y}$ is affine. Additionally, $\pi$ is a good moduli space, so cohomologically affine. It follows from \cite[Proposition 3.13]{Alper} that $f$ is cohomologically affine. Since $f$ is representable, by Proposition 3.3 of (loc.~cit), $f$ is affine.

Next, we prove $f$ sends closed points to closed points. If $f$ is surjective, this is \cite[Lemma 6.5]{Alper}. If $Y$ is proper over $S$ and $y$ is a closed point of $\cY$, its image $\pi(y)\in Y$ is closed since $\pi$ is universally closed by Theorem 4.16 of (loc.~cit). As $Y$ is proper over $S$, the image $s\in S$ of $\pi(y)$ is also closed. To check $f(y)$ is closed, it therefore suffices to check this after base changing to the fiber over $s$, where the result holds by our hypothesis on $f_s$.

Lastly, having shown $f$ is representable, separated, and maps closed points to closed points, \cite[Proposition 6.4]{Alper} (which uses the locally Noetherian hypothesis) shows that $f$ is finite; we apply the proposition with $g=\id_Y$. 
\end{proof}


One of the inputs to Proposition \ref{prop:finiteness-property--->finite} is knowing $f$ is quasi-finite. We show that in our case of interest, quasi-finiteness is an open condition on the base.

\begin{lemma}\label{l:qf-is-open}
Let $S$ be a scheme and $f\colon\cY\to\cY'$ be a representable morphism of algebraic $S$-stacks which is locally of finite type. Suppose the structure map $g\colon\cY\to S$ is universally closed. Then there is an open subscheme $U\subset S$ with the following properties:
\begin{enumerate}
\item\label{l:qf-is-open::1} for all $s\in U$, the map $f_s$ is quasi-finite, and
\item\label{l:qf-is-open::2} if $S'\to S$ is a morphism of schemes and if the base change $f_{S'}\colon\cY_{S'}\to\cY'_{S'}$ has $f_{s'}$ quasi-finite for all $s'\in S'$, then there is a unique map $S'\to U$ of $S$-schemes for which $f_{S'}=f_U\times_U {S'}$.
\end{enumerate}
\end{lemma}
\begin{proof}
By fqpc descent, quasi-finiteness is a condition on geometric fibers. As a result, properties (\ref{l:qf-is-open::1}) and (\ref{l:qf-is-open::2}) will follow immediately upon showing that
\[
U:=\{s\in S\mid f_s\textrm{\ is\ quasi-finite}\}
\]
is an open subscheme of $S$.

By definition, $S\smallsetminus U$ is the set of $s$ such that there exists $y\in|\cY|$ with $g(y)=s$ as elements of $|S|$ and $\dim_y f^{-1}(f(y))\neq0$. In other words, if
\[
\cV=\{y\in|\cY|\mid \dim_y f^{-1}(f(y))=0\},
\]
then $g(\cY\smallsetminus\cV)=S\smallsetminus U$. Since $g$ is closed, it is therefore enough to show $\cV$ is open.

Let $\widetilde{Y}'\to\cY'$ be a smooth cover by a scheme and let
\[
\xymatrix{
\widetilde{Y}\ar[r]^-{\widetilde{f}}\ar[d]_-{p} & \widetilde{Y}'\ar[d]\\
\cY\ar[r]^-{f} & \cY'
}
\]
be the pullback. 
Let
\[
\widetilde{V}=\{\widetilde{y}\in|\widetilde{Y}|\mid \dim_{\widetilde{y}} \widetilde{f}^{-1}(\widetilde{f}(\widetilde{y}))=0\},
\]
which is an open subset by \cite[Corollary 12.79]{GortzWedhorn}. Then $\cV=p(\widetilde{V})$ and since $p$ is open, we see $\cV$ is as well.
\end{proof}

We turn now to our first goal of this section.

\begin{proposition}\label{prop:iso-in-hom-open}
Let $\cT$ and $\cT'$ be warps of a locally Noetherian scheme $T$. Then the substack $\uIsom_T(\cT,\cT')$ parametrizing isomorphisms $\phi\in\uHom_T(\cT,\cT')$ is open.\end{proposition}

\begin{proof}
By Proposition \ref{prop:rep-open-gms}, after replacing $T$ by an open subscheme, we may assume $\phi$ is representable. Next, by \cite[Tag 05XC]{stacks-project}, we may assume $\phi$ is flat and surjective; note that the result is stated for morphisms of algebraic spaces but the proof applies equally well to representable maps of algebraic stacks. Lemma \ref{l:qf-is-open} tells us we may further assume that $\phi$ is quasi-finite. Since $T$ is locally Noetherian, we see $\cT$ and $\cT'$ are as well. Then applying Proposition \ref{prop:finiteness-property--->finite}(\ref{prop:finiteness-property--->finite::main}) with $S=Y=T$ implies such $\phi$ are finite.

Since $\phi$ finite, we see $\cT\times_{\cT'}\cT\to\cT'$ finite and therefore $\cT\times_{\cT'}\cT\to T$ is universally closed. Then \cite[Tag 05X9]{stacks-project} tells us we may assume $\phi$ is finite, unramified, and universally injective, hence a closed immersion by Tag 05W8 of (loc.~cit). As $\phi$ is flat and finitely presented, it is open. Since $\phi$ is also surjective, it is an isomorphism.
\end{proof}

We next prove a local version of Theorem \ref{thm:sep-diag-coh-affine}.

\begin{proposition}\label{prop:localthm:sep-diag-coh-affine}
If $\cX$ is cohomologically affine, then the diagonal of $\cW(\cX)$ is quasi-affine.
\end{proposition}
\begin{proof}
Let $T$ be an affine scheme, $\cX_T:=\cX\times_T k$, and $\gamma:=(\pi\colon\cT\to T,f\colon\cT\to\cX_T)$ and $\gamma':=(\pi'\colon\cT'\to T,f'\colon\cT'\to\cX_T)$ be $T$-points of $\cW(\cX)$. We must prove that the sheaf $\uIsom_T(\gamma,\gamma')$ is quasi-affine over $T$.

Since Theorem \ref{thm:main--A0-algebraic}(\ref{thm:A0-algebraic}) shows $\cW(\cX)$ is locally of finite presentation, by Noetherian induction, we may assume $T$ is locally Noetherian. 
By \cite[Proposition 3.10(vi)]{Alper}, $\cX_T\to T$ is cohomologically affine, as $\cX\to\Spec k$ is. Since $\Delta_{\cX_T/T}$ and $\Delta_{T/k}$ are affine, Proposition 3.13 of (loc.~cit) shows $f$ and $f'$ are cohomologically affine, hence affine since they are representable. By \cite[Theorem 5.10]{AHRLuna}, the Hom-stacks $\uHom_T(\cT,\cT')$ and $\uHom_T(\cT,\cX_T)$ are algebraic. Consider the fiber product
\[
\xymatrix{
V\ar[r]\ar[d] & T\ar[d]^-{F}\\
\uHom_T(\cT,\cT')\ar[r]^-{h} & \uHom_T(\cT,\cX_T)
}
\]
where $F$ is given by the map $f\colon\cT\to\cX_T$ and $h=\uHom_T(-,f')$. It suffices to show $h$ is affine. Indeed, upon showing this, we know $V\to T$ is affine and, from its definition, we see $V$ parameterizes pairs $(\phi,\alpha)$ with $\phi\colon\cT\to\cT'$ a $T$-morphism and $\alpha\colon f\Rightarrow f'\phi$ a $2$-isomorphism. Additionally, a morphism $(\phi,\alpha)\to(\phi',\alpha')$ in the fiber product is given by a $2$-isomorphism $\beta\colon\phi\Rightarrow\phi'$ such that $f'(\beta)\circ\alpha=\alpha'$; such a $\beta$ is unique if it exists. By Proposition \ref{prop:iso-in-hom-open}, the locus $\uIsom_T(\cT,\cT')$ is open in $\uHom_T(\cT,\cT')$ and we therefore see
\[
\uIsom_T(\gamma,\gamma')=\uIsom_T(\cT,\cT')\times_{\uHom_T(\cT,\cT')} V,
\]
which is open in $V$, hence quasi-affine over $T$.

We turn now to the claim that $h$ is affine. Letting $p\colon\cZ=\cT'\times_{\cX_T} \cT\to \cT$ be the projection, we see $p$ is affine and so $\cZ=\uSpec\cF$ for some quasi-coherent $\cO_\cT$-algebra $\cF$. Note that for any $U\to T$, morphisms $\cT_U\to\cT'_U$ mapping to $f'$ under $h$ are as sections of $p$. Thus, the $U$-points of $V$ are given by maps $q\colon U\to T$ and $\cO_U$-algebra homomorphisms $q^*\cF\to q^*\cO_\cT=\cO_U$ giving a section of the structure map $\cO_U\to q^*\cF$. Since $\cT\to T$ is a flat good moduli space map, \cite[Corollary 5.15]{AHRLuna} tells us $\uHom_{\cO_\cT/T}(\cF,\cO_\cT)$, parameterizing $\cO$-module homorphisms, is affine over $T$; the proof of \cite[Theorem 2.3(i)]{HallRydh-generalHilbert} then 
shows that $V$ is closed in $\uHom_{\cO_\cT/T}(\cF,\cO_\cT)$, hence affine over $T$.
\end{proof}

\begin{proof}[{Proof of Theorem \ref{thm:sep-diag-coh-affine}}]
First, quasi-affineness of the map $I_{\cW(\cX)}\to\cW(\cX)$ implies the geometric stabilizers of $\cW(\cX)$ quasi-affine algebraic groups, hence affine by \cite[VIB.11.1]{SGA3}. 

Next, \cite[Tag 0CL0]{stacks-project} shows that $I_{\cW(\cX)}\to\cW(\cX)$ is separated if and only if $\Delta_{\cW(\cX)/k}$ is separated. Combining this with Tag 06R5 of (loc.~cit), we see that to check $\Delta_{\cW(\cX)/k}$ is separated, it suffices to look Zariski locally on $\cW(\cX)$; in particular, we may look Zariski locally on $Y$. If $U\hookrightarrow Y$ is an open affine subscheme, and $\cV:=\cX\times_Y V$, then Proposition \ref{prop:X-gms->A0Xgms}(\ref{AoXbaseChangesWell}) shows $\cW(\cV)=\cW(\cX)\times_Y V$. As $\cW(\cV)$ is cohomologically affine over $V$, hence over $k$, Proposition \ref{prop:localthm:sep-diag-coh-affine} shows $\Delta_{\cW(\cV)/k}$ is separated. 
\end{proof}

\section{Warped arcs with integrity}
\label{subsec:finiteness-property}

Our main goal in this section is to prove:

\begin{proposition}\label{prop:openness-geom-fibers-for-integrity}
Let $\cX$ be an Artin stack locally of finite presentation over $k$ with affine diagonal and admitting a good moduli space. Let $\gamma\in\cW(\cX)(T)$ with $T$ a locally Noetherian scheme.

Then there is an open subscheme $W\subset T$ such that a map $g\colon T'\to T$ factors (uniquely) through $W$ if and only if $g^*\gamma$ has integrity.
Furthermore, $W$ is characterized by the condition that $w^*\gamma$ has integrity for all geometric points $w$ of $W$.
\end{proposition}

We begin with the following useful observation.

\begin{remark}\label{rmk:sharp-commutes-with-bc}
Let $f\colon\cY\to\cZ$ be a morphism of Artin stacks over a scheme $T$ and keep the notation from Definition \ref{def:integrity}. Then $f^\sharp$ commutes with arbitrary base change on $T$. Indeed, if $T'\to T$ is a map, then since good moduli spaces commute with base change \cite[Proposition 4.7]{Alper} and since the map to a good moduli space is a universal map to algebraic spaces \cite[Theorem 3.12]{AHRetalelocal}, we find $\overline{f}\times_T T'=\overline{f\times_T T'}$. As a result, $f^\sharp\times_T T'=(f\times_T T')^\sharp$.
\end{remark}

We first show that integrity for a total family is equivalent to each geometric fiber having integrity.

\begin{lemma}\label{l:fin-int-on-fibers}
Let $\cX$ be an Artin stack locally of finite presentation over $k$ with affine diagonal, admitting a good moduli space. Let $(\cT\to T,f\colon\cT\to\cX)\in\cW(\cX)(T)$ with $T$ a locally Noetherian scheme. Then $f$ has integrity if and only if for all geometric points $t$ of $T$, the map $f_t\colon\cT_t\to\cX$ has integrity.
\end{lemma}
\begin{proof}
Let $\cX\to Y$ be a good moduli space. Let $\cY:=\cX\times_Y T$ and $f^\sharp\colon\cT\to\cY$ be the induced (representable) map. By Remark \ref{rmk:sharp-commutes-with-bc}, we must show $f^\sharp$ is finite if and only if its base change $f^\sharp_t\colon\cT_t\to\cY_t$ is finite for all geometric points $t$ of $T$. One direction is clear, so we may assume $f^\sharp_t$ is finite for all geometric points and show $f^\sharp$ is finite. By fpqc descent, see e.g., \cite[Tag 02LA]{stacks-project}, we know $f^\sharp_t$ is finite for all fibers $t\in T$. 

Since $T$ is locally Noetherian, so are $\cT$ and $\cY$. Note that $\cY\to T$ is a good moduli space map 
and $\Delta_{\cY/T}$ is affine. Since we are assuming $f^\sharp_t$ is finite for all $t\in T$, we see $f^\sharp_t$ maps closed points to closed points and $f^\sharp$ is quasi-finite. Therefore, applying Proposition \ref{prop:finiteness-property--->finite}(\ref{prop:finiteness-property--->finite::main}) to $f^\sharp$ with $S=T$, we see $f^\sharp$ is finite.
\end{proof}

Our next goal is to show that for warped arcs, integrity is open condition on the base. We do this after a preliminary result.


\begin{proposition}\label{prop:finite-open-condition}
Let $T$ be a scheme and let
\[
\xymatrix{
\cY\ar[r]^-{f}\ar[dr]_-{\pi} & \cY'\ar[d]^-{\pi'}\\
 & T
}
\]
be a commutative diagram Artin stacks, where $\pi$ and $\pi'$ are good moduli space maps. Suppose $\Delta_{\cY'/T}$ is affine, and that $f$ is representable, 
and of finite type. Then there is an open subscheme $U\subset T$ with the following properties:
\begin{enumerate}
\item\label{l:fin-is-open::1} for all $t\in U$, the map $f_t$ is finite, and
\item\label{l:fin-is-open::2} if $T'\to T$ is a morphism of schemes and if the base change $f_{T'}\colon\cY_{T'}\to\cY'_{T'}$ has $f_{t'}$ finite for all $t'\in S'$, then $T'\to T$ factors (uniquely) through $U$.
\end{enumerate}
\end{proposition}
\begin{proof}
If $T'\to T$ is a morphism of schemes and $t'\in T'$, then by fpqc descent, $f_{t'}$ finite if and only if it is finite over any geometric point mapping to $t'$. Thus, parts (\ref{l:fin-is-open::1}) and (\ref{l:fin-is-open::2}) are equivalent to showing that the locus of $t\in T$ with $f_t$ is open. 

First, by Lemma \ref{l:qf-is-open}, we may assume $f$ is quasi-finite. Next, by our assumption on the diagonal, \cite[Proposition 3.13]{Alper} shows that $f$ is cohomologically affine. Since $f$ is representable, it is affine. Then applying Zariski's Main Theorem, we see that $f$ factors as
\[
\cY\xhookrightarrow{j}\cV\xrightarrow{p}\cY',
\]
where $p$ is finite, $j$ is an open immersion, and the natural map $\cO_\cV\hookrightarrow j_*\cO_\cY$ is an injection. Since $f$ is affine and $p$ is separated, we see $j$ is affine as well. By Lemma 4.14 of (loc.~cit) and universality of good moduli spaces \cite[Theorem 3.12]{AHRetalelocal}, we have a commutative diagram
\[
\xymatrix{
\cY\ar@{^{(}->}[r]^-{j}\ar[d]_-{\pi} & \cV\ar[r]^-{p}\ar[d]_-{\phi} & \cY'\ar[d]^-{\pi'}\\
T\ar@{^{(}->}[r]^-{i} & V\ar[r]^-{q} & T
}
\]
where $V=\uSpec(\pi'_*p_*\cO_\cV)$ and $q\circ i=\id_T$. Since $q$ is affine (hence separated) and $q\circ i$ is an isomorphism, it follows that $i$ is a closed immersion. Furthermore, the natural map $\cO_V\hookrightarrow i_*\cO_T$ is an injection, hence an isomorphism, i.e., $i$ and $q$ are isomorphisms.

Now, since $p$ is finite, we see that for any $t\in T$, the map $f_t$ is finite if and only if $j_t$ is finite. Since $j_t$ is quasi-finite, affine, and of finite type, \cite[Tag 02LS]{stacks-project} shows $j_t$ is finite if and only if it is universally closed; as $j_t$ is an open immersion, this is equivalent to surjectivity of $j_t$. Since $\phi$ is closed, $Z:=\phi(\cV\setminus j(\cY))$ is closed. Its complement $V\setminus Z$ is the set of $v\in V$ where $j_v$ is surjective. As $q$ is an isomorphism, this finishes the proof. 
\end{proof}

\begin{proof}[{Proof of Proposition \ref{prop:openness-geom-fibers-for-integrity}}]
Let $\cY$ and $f^\sharp$ be as in the proof of Lemma \ref{l:fin-int-on-fibers}. Since $f$ is representable, $f^\sharp$ is as well. Proposition \ref{prop:finite-open-condition} then shows that the locus $W\subset T$ of $t\in T$ where $f^\sharp_t$ is finite forms an open subscheme. Since integrity is a condition that can be checked on geometric fibers by Lemma \ref{l:fin-int-on-fibers}, we see if $g\colon T'\to T$ is a map with $g^*\gamma$ having integrity, then $g$ factors through $W$.
\end{proof}

\section{The valuative criterion: proof of Theorem \ref{thm:val-crit}}

\begin{proof}[{Proof of Theorem \ref{thm:val-crit}}]
Let $k'/k$ be a field extension and recall that from Theorem \ref{thm:main--A0-algebraic}(\ref{thm:A0-algebraic}) (see Remark \ref{rmk:LA0Xk-vs-A0XD}), the category $\sL(\cW(\cX))(k')$ is naturally identified with $\cW(\cX)(k'[[t]])$; additionally recall that Remark \ref{rmk:X-gms->havemapA0X->Y} yields a map 
\[
\Lambda\colon\cC_\cX\hookrightarrow|\sL(\cW(\cX))|\setminus|\sL(\cW(\cX\setminus\cU))|\to|\sL(Y)|\setminus|\sL(Y\setminus U)|
\]
which we wish to show is bijective.

Let $D=\Spec k'[[t]]$ and $\varphi^{\can}=(\pi\colon\cD\to D,f\colon\cD\to\cX)$. By Proposition \ref{prop:can-lift-toArtin-arc} and Remark \ref{rmk:can-lift-toArtin-arc}, we see $\varphi^{\can}\in\cC_\cX(k')\subset\sL(\cW(\cX))(k')$. 
Since $\Lambda(\varphi^{\can})=\varphi$, we have shown surjectivity of $\Lambda$ on $k'$-points.

Now suppose $\gamma\in\cC_\cX(k')$ with $\Lambda(\gamma)=\varphi$, i.e., $\gamma=(\pi'\colon\cD'\to D,f'\colon\cD'\to\cX)$, 
$\cD'$ is normal, $\varphi$ is the map induced by $f'$ on good moduli spaces, $f'$ has integrity, and $f'$ maps the generic fiber of $\pi'$ to $\cU$. By Proposition \ref{prop:twistede-discs-factor-through-can-lift-toArtin-arc}, this therefore yields a commutative diagram
\[
\xymatrix{
\cD'\ar[r]^-{g}\ar[dr]_-{\pi'}\ar@/^1.3pc/[rr]^-{(f')^\sharp} & \cD\ar[d]^-{\pi}\ar[r]^-{f^\sharp} & \cX\times_Y D\ar[dl]^-{p'}\\
& D & 
}
\]
where $g$ is unique up to unique isomorphism. As $f^\sharp$ is representable, $g$ is as well. Furthermore, $g$ is finite since $f^\sharp$ and $(f')^\sharp$ are. Then $g$ is a finite birational map of normal stacks, so it is an isomorphism by Zariski's Main Theorem.
\end{proof}

Lastly, we show that for separated good moduli spaces (e.g., when $\cX$ has finite inertia), we recover all twisted arcs as examples of warped arcs.

\begin{proposition}\label{prop:recover-twisted arcs}
With notation and hypotheses as in Theorem \ref{thm:val-crit}, if $\cX\to Y$ is separated, then $\varphi^{\can}$ is a twisted arc.
\end{proposition}
\begin{proof}
By Proposition \ref{prop:twisted arcs-integrity}, we know every twisted arc $\gamma=(\pi\colon\cD\to D,f\colon\cD\to\cX)$ lifting $\varphi$ has integrity. Since $\cD$ is normal by definition, we see $\gamma\in\cC_\cX(k')$ and maps to $\varphi$ under the morphism $\cC_\cX\to|\sL(Y)|\setminus|\sL(Y\setminus U)|$. By the valuative criterion (Theorem \ref{thm:val-crit}), we therefore have $\gamma=\varphi^{\can}$.
\end{proof}

\begin{remark}
In the Deligne--Mumford case, taking fiber products and normalizations to define canonical lifts have appeared in other work as well; see e.g., \cite[Proposition 2.5]{ESZB} and \cite[Proposition 2.12]{GroechenigWyssZiegler}. 
\end{remark}

\section{The canonical cylinder:~proof of Theorem \ref{thm:main-canonical-cylinder}}

We keep the set-up in Notation \ref{not:main}. In this section, we introduce the canonical cylinder $\cC'_\cX\subset|\sL(\cW(\cX))|$ and characterize (up to measure zero), the warped arcs in $\cC'_\cX$. Let
\[
\theta_n\colon \sL(\cW(\cX))\to\sL_n(\cW(\cX))
\]
denote the truncation map. By Theorem \ref{thm:main--A0-algebraic}(\ref{thm:A0-algebraic}) (see Remark \ref{rmk:LA0Xk-vs-A0XD}), we may identity $\sL(\cW(\cX))(k')$ with $\cW(\cX)(k'[[t]])$.

\begin{definition}\label{def:can-cylinder}
If $\cX$ admits a good moduli space, we define the \emph{canonical cylinder}
\[
\cC'_\cX\subset|\sL(\cW(\cX))|
\]
to be all isomorphism classes of arcs $\gamma$ such that $\theta_1(\gamma)$ is $(R_1)$-aspiring and $\theta_0(\gamma)$ is $(S_1)$ and has integrity, see Definitions \ref{def:S2-artinArc}, \ref{def:aspiringR1-artinArc}, and \ref{def:integrity}.
\end{definition}

\begin{theorem}\label{thm:main-can-cylinder}
If $\cX$ admits a good moduli space, then  
\begin{enumerate}
\item\label{thm:main-can-cylinder::cylinder} $\cC'_\cX$ is a cylinder.

\item\label{thm:main-can-cylinder::characterization} Let $\gamma=(\pi\colon\cD\to D, f\colon\cD\to\cX)\in\sL(\cW(\cX))(k')$ with $k'$ be a field and $D=\Spec k'[[t]]$. If $\pi$ an isomorphism over the generic point, then $\gamma\in\cC'_\cX$ if and only if $\cD$ is normal and $f$ has integrity, see Definition \ref{def:integrity}.
\end{enumerate}
\end{theorem}


\subsection{Serre's property $(S_2)$}\label{sec:S2-property}
In this subsection, we give a criterion that will guarantee a warped disc $\cD$ satisfies Serre's property $(S_2)$. Recall that $(S_r)$ is defined for stacks since this is a property one may check on smooth covers by Tags 036A and 0226 of \cite{stacks-project}.

\begin{lemma}\label{l:S2-property-on-level0}
Let $r\geq1$ and $T$ be a finite type scheme over a field and let $\cD\to T$ be a warp. 
Then the following hold:
\begin{enumerate}
\item\label{l:S2-property-on-level0::constructible} The set of $t\in T$ where the fiber $\cD_t$ is $(S_r)$ forms a constructible subset.

\item\label{l:S2-property-on-level0::DVR} Let $T=\Spec R=:D$ with $R$ a DVR with uniformizer $t$, $D_0=\Spec R/t$, and consider the cartesian diagram
\[
\xymatrix{
\cD_0\ar@{^{(}->}[r]\ar[d]_-{\pi_0} & \cD\ar[d]^-{\pi}\\
D_0\ar@{^{(}->}[r] & D
}
\]
where $\pi$ is a warped disc. Then $\cD$ is $(S_r)$ if and only if $\cD_0$ is $(S_{r-1})$.

\item\label{l:S2-property-on-level0::geom-pts} Let $k''/k'$ be an extension of fields and let $T=\Spec k'$. Then $\cD$ is $(S_r)$ if and only if $\cD\times_{k'}k''$ is.
\end{enumerate}
\end{lemma}
\begin{proof}
For part (\ref{l:S2-property-on-level0::constructible}), it suffices to replace $\cD$ by a smooth cover, in which case the result follows from \cite[(9.9.3)]{EGA43}. 

For part (\ref{l:S2-property-on-level0::DVR}), let $V\to\cD$ be a smooth cover and let $V_0=V\times_\cD\cD_0$. Since $V\to D$ is flat, we see $V_0\subset V$ is a closed subscheme defined by a non-zero divisor. So, $V$ is $(S_r)$ if and only if $V_0$ is $(S_{r-1})$, and the result follows.

Lastly, for part (\ref{l:S2-property-on-level0::geom-pts}), we may replace $\cD$ by a smooth cover $V\to\cD$ to assume it is a scheme which is locally finitely presented over $k'$. Then $V$ is $(S_r)$ if and only if $V\times_{k'}k''$ is. Indeed, the only if direction is given by \cite[Tag 0352]{stacks-project} and the if direction follows from \cite[Tag 0337]{stacks-project}.
%
\end{proof}

For brevity, we introduce the following definition.

\begin{definition}\label{def:S2-artinArc}
Let $k'$ be a field. We say $(\cD_0\to\Spec k',\cD_0\to\cX)\in\cW(\cX)(k')$ \emph{is $(S_r)$} if $\cD_0$ is $(S_r)$.
\end{definition}

\subsection{Regularity in codimension $1$}
\label{subsec:R1}

We next come to a property that guarantees a warped disc $\cD\to D$ is regular in codimension $1$, i.e., satisfies property $(R_1)$. Unlike integrity and the $(S_r)$-property, this next property is \emph{not} determined by information over the closed point of $D$, but rather its first order neighborhood.

Recall from \cite[Definition 3.1]{SatrianoUsatine2} that if $\cY$ is a finite type Artin stack over a field $k'$ and $E\in D^-_{\coh}(\cY)$, then the $0$th-truncated height function $\het^{(i)}_{0,E}\colon|\cY|\to\Z_{\geq0}$ is defined by
\[
\het^{(i)}_{0,E}(y)=\dim_K L^iy^*E,
\]
where $K$ is a field and $y\colon\Spec K\to\cY$ is any map representing the point $y$. Let
\[
h_E:=\het^{(0)}_{0,E}-\het^{(1)}_{0,E}
\]
and
\[
\cZ^{\geq d}(E):=\{y\mid h_E(y)\geq d\}\subset |\cY|.
\]

\begin{definition}\label{def:LD-dim-d-in-codim-c}
We say $\cY$ has \emph{$E$-dimension at most $d$ in codimension $c$} if every irreducible component of $\cZ^{\geq d+1}(E)$ has codimension at least $c+1$. 
\end{definition}

\begin{lemma}\label{l:E-ht-d-c-constructible}
Let $\cY\to T$ be a finitely presented morphism with $\cY$ an algebraic stack and $T$ a finitely presented scheme over a field $k'$. If $E\in D^-_{\coh}(\cY)$, then $\cZ^{\geq d}(E)$ is constructible for all $d$.

Furthermore, if $T=\Spec k'$ and $\rho\colon V\to\cY$ is a finitely presented smooth cover of pure relative dimension $r$ (e.g., if $\cY$ is a global quotient stack of $V$), then $\cZ^{\geq d}(L_{\cY/k'})=\rho(\cZ^{\geq d+r}(\Omega^1_{V/k'}))$ is closed and 
\begin{equation}\label{eqn:Lht-Omegaht}
h_{L_{\cY/k'}}(y)=\dim_K (v^*\Omega^1_{V/k'}) - r
\end{equation}
for any $v\colon\Spec K\to V$ with $y=\rho(v)$.
\end{lemma}
\begin{proof}
For any $c$, let $\cZ^{(i)}_{\geq c}$ be the set of $y\in|\cY|$ with $\het^{(i)}_{0,E}(y)\geq c$; this locus is closed by \cite[Theorem 3.2(a)]{SatrianoUsatine2}, and so the level sets $\cZ^{(i)}_{c}$ of $\het^{(i)}_{0,E}$ are locally closed. Then $\cZ^{\geq d}(E)$ is the disjoint union of all $\cZ^{(1)}_c\cap\cZ^{(0)}_{\geq d+c}$. Since $\cY$ is finite type, this can be written as a finite disjoint union, hence $\cZ^{\geq d}(E)$ is constructible.

Now assume $\rho\colon V\to\cY$ is a finitely presented smooth cover of pure relative dimension $r$. Since $E=\Omega^1_{V/k'})$ is a complex concentrated in degree $0$, we see $h_E(v)=\dim_K v^*\Omega^1_{V/k'}$ which is an upper semi-continuous function by \cite[tag 0BDI]{stacks-project}. As a result, $\cZ^{\geq d}(\Omega^1_{V/k'})$ is closed. Next, from the exact triangle
\[
\rho^*L_{\cY/k'}\to L_{V/k'}\to \Omega^1_{V/\cY},
\]
and the fact that $v^*\Omega^1_{V/\cY}=L^0v^*_{V/\cY}$, we obtain an exact sequence
\[
0\to L^0y^*L_{\cY/k'}\to v^*\Omega^1_{V/k'}\to v^*\Omega^1_{V/\cY}\to L^1y^*L_{\cY/k'}\to 0.
\]
Since $\dim_K v^*\Omega^1_{V/\cY}=r$, we have shown \eqref{eqn:Lht-Omegaht}. It then follows that $\cZ^{\geq d}(L_{\cY/k'})=\rho(\cZ^{\geq d}(\Omega^1_{V/k'}))$. 
\end{proof}

We turn now to our main case of interest:~$1$-jets of $\cW(\cX)$ where $E$ is given by the cotangent complex, $d=1$, and $c=0$. Since we show such $1$-jets yield warped discs that are $(R_1)$, we introduce the following definition.

\begin{definition}\label{def:aspiringR1-artinArc}
Let $k'$ be a field and $D_{1,k'}=\Spec k'[t]/t^2$. We say $(\cD_1\to D_{1,k'},f_1\colon\cD_1\to\cX)\in\sL_1(\cW(\cX))(k')$ \emph{is $(R_1)$-aspiring} if $\cD_1$ has $L_{\cD_1/k'}$-dimension at most $1$ in codimension $0$.
\end{definition}

\begin{proposition}\label{prop:R1-property}
Let $k'/k$ be a field extension, $T$ be a finitely presented $k'$-scheme, $D=\Spec k'[[t]]$, and $D_{n,T}=T\times_{k'}\Spec k'[t]/(t^{n+1})$. Fix $d,c\in\Z_{\geq0}$. Then
\begin{enumerate}
\item\label{l:R1-level1::constructible} Let $(\cD_1\to D_{1,T},\cD_1\to\cX)\in\cW(\cX)(D_{1,T})$. Then the set of $t\in T$ where $\cD_{1,t}$ has $L_{\cD_{1,t}/k(t)}$-dimension at most $d$ in codimension $c$ forms a constructible subset of $T$.

\item\label{l:R1::DVR} Let $\gamma=(\pi\colon\cD\to D,\cD\to\cX)\in\cW(\cX)(D)$ and $\gamma_1=(\cD_1\to D_{1,k'},\cD_1\to\cX)\in\cW(\cX)(D_{1,k'})=\sL_1(\cW(\cX))(k')$ be the induced object. If $\pi$ is an isomorphism over the generic point, then $\gamma_1$ is $(R_1)$-aspiring if and only if $\cD$ is $(R_1)$.

\item\label{l:R1::geom-pts} If $k''/k'$ is a field extension, then $\gamma_1\in\sL_1(\cW(\cX))(k')$ is $(R_1)$-aspiring if and only if $\gamma_1\times_{k'}k''$ is.
\end{enumerate}
\end{proposition}
\begin{proof}
For part (\ref{l:R1-level1::constructible}), flatness of the map $\cD_1\to D_{1,T}$ implies $L\iota_t^*L_{\cD_1/T}=L_{\cD_{1,t}/k(t)}$ for all $t\colon\Spec k(t)\to T$. As a result, the fiber of $\cZ(L_{\cD_1/T})$ over $t$ is given by $\cZ(L_{\cD_{1,t}/k(t)})$. By \cite[Corollary 4.14(1)]{AHRLuna}, we know $\cD_{1,t}$ is a global quotient stack and so $\cZ(L_{\cD_{1,t}/k(t)})$ is closed by Lemma \ref{l:E-ht-d-c-constructible}. The same lemma shows that $\cZ(L_{\cD_1/T})$ is constructible. Part (\ref{l:R1-level1::constructible}) then follows from \cite[(9.9.5)]{EGA43}.

For part (\ref{l:R1::geom-pts}), let $\gamma_1=(\cD_{1,k'}\to D_{1,k'},\cD_{1,k'}\to\cX)$. Since $k''/k'$ is flat, the base change map $p\colon\cD_{1,k''}=\cD_{1,k'}\times_{k'}k''\to \cD_{1,k'}$ is codimension-preserving and $p^*L_{\cD_{1,k'}/k'}=L_{\cD_{1,k''}/k''}$. It follows that $\cZ^{\geq 2}(L_{\cD_{1,k''}/k''})=p^{-1}(\cZ^{\geq 2}(L_{\cD_{1,k'}/k'}))$, so $\cZ^{\geq 2}(L_{\cD_{1,k''}/k''})$ has codimension at least $1$ if and only if $(\cZ^{\geq 2}(L_{\cD_{1,k'}/k'})$ does.

For part (\ref{l:R1::DVR}), we again apply \cite[Corollary 4.14(1)]{AHRLuna} to show $\cD=[V/\GL_n]$ where $V$ is an affine scheme. 
Let $\cD_n=\cD\times_D D_{n,k'}$, $V_n=V\times_\cD \cD_n$, and $\rho_n\colon V_n\to\cD_n$ be the induced cover. We see $\cD$ is irreducible of dimension $1$ as $\pi\colon\cD\to D$ is an isomorphism over the generic point and $\cD$ is the closure of its generic fiber. As $\GL_n$ is geometrically irreducible, it follows that $V$ is irreducible of dimension $n^2$. Since $V_0\subset V$ is defined by a non-zero divisor, we see $V_0$ is equidimensional, and hence $\cD_0=[V_0/\GL_n]$ is too with $\dim\cD_0=0$. 

Next, $\cD$ is regular in codimension $1$ if and only if $V$ is. Since $\pi$ is an isomorphism over the generic point, $V$ is regular in codimension $1$ if and only if it is regular at all codimension $1$ points that lie over the special fiber, i.e., at all generic points of the specal fiber. Furthermore, since $V\to D$ is of finite type, Tags 07P7 and 07PJ of \cite{stacks-project} tell us the regularity locus in $V$ is open, and hence, $V$ is regular in codimension $1$ if and only if there exists a closed subscheme $Z\subset V_0$ of codimension at least $1$ such that $V\setminus Z$ is regular.

By flatness of $V\to D$, we see $V_1$ is the first infinitesimal neighborhood of $V_0$ in $V$. In particular, letting $v\colon\Spec k(v)\to V_0$ be the map from the residue field of a point and letting $\fm_{V,v}\subset\cO_{V,v}$ (resp.~$\fm_{V_1,v}\subset\cO_{V_1,v}$) denote the maximal ideals in the corresonding local rings, we obtain a factorization $\Spec\cO_{V,v}/\fm_{V,v}^2\to V_1$, which induces an isomorphism $\fm_{V,v}/\fm_{V,v}^2\xrightarrow{\simeq} \fm_{V_1,v}/\fm_{V_1,v}^2$. Now, $V$ is regular at $v$ if and only if $\dim_{k(v)}\fm_{V,v}/\fm_{V,v}^2\leq\dim V$, or equivalently $\dim_{k(v)}\fm_{V_1,v}/\fm_{V_1,v}^2\leq\dim V_1+1$. Since the regularity locus of $V$ is open, we need only consider closed points $v$, in which case $k(v)/k'$ is algebraic as $V_1\to\Spec k'$ is of finite presentation. Applying \cite[Tag 0B2E]{stacks-project}, we have a natural isomorphism
\[
\fm_{V_1,v}/\fm_{V_1,v}^2\xrightarrow{\simeq} v^*\Omega^1_{V_1/k'}. 
\]
It follows from Lemma \ref{l:E-ht-d-c-constructible} that $V$ is not regular at $v$ if and only if $\rho_1(v)\in\cZ:=\cZ_{\geq2}(L_{\cD_1/k'})$. Therefore, $V$ is regular in codimension $1$ if and only if $\cZ_{\geq2}(L_{\cD_1/k'})$ has codimension at least $2$ in $\cD$ or equivalently, codimension at least $1$ in $\cD_1$.
\end{proof}

\subsection{Proof of Theorems \ref{thm:main-canonical-cylinder} and \ref{thm:main-can-cylinder}}\label{sec:thm:main-can-cylinder}

We begin with Theorem \ref{thm:main-can-cylinder}.

\begin{proof}[{Proof of Theorem \ref{thm:main-can-cylinder}}]
For part (\ref{thm:main-can-cylinder::cylinder}), by construction, $\cC'_\cX$ is defined by conditions on the zeroth and first truncations of $k'$-points of $\sL(\cW(\cX))$, so it remains to prove that these conditions are locally constructible. Let $\cC_{\textrm{fin}}\subset|\sL_0(\cW(\cX))|$, $\cC_{S_1}\subset|\sL_0(\cW(\cX))|$, and $\cC_{R_1}\subset|\sL_1(\cW(\cX))|$ be the loci defined by warped arcs which have integrity, respectively are $(S_1)$, respectively are $(R_1)$-aspiring; these loci are well-defined by Proposition \ref{prop:openness-geom-fibers-for-integrity}, Lemma \ref{l:S2-property-on-level0}(\ref{l:S2-property-on-level0::geom-pts}), and Proposition \ref{prop:R1-property}(\ref{l:R1::geom-pts}). Note that
\[
\cC'_\cX=\cC_{\textrm{fin}}\cap\cC_{S_1}\cap\cC_{R_1}.
\]
For brevity, we let $P\in\{\textrm{fin},S_1,R_1\}$ and let $n(\textrm{fin})=n(S_1)=0$ and $n(R_1)=1$. By Theorem \ref{thm:main--A0-algebraic}(\ref{thm:A0-algebraic}), each $\sL_{n(P)}(\cW(\cX))$ is locally of finite type over $k$, hence local constructibility of $\cC_P$ in $|\sL_{n(P)}(\cW(\cX))|$ is equivalent to constructibility of $\cV\cap\cC_P$ for any finite type open substack $\cV\subset\sL_{n(P)}(\cW(\cX))$.

Fix such a finite type open substack $\cV$ and let $\rho\colon T\to\cV$ be a finite type smooth cover by a scheme. Let $\gamma=(\pi\colon\cD\to T, f\colon\cD\to\cX)$ be the warped map defined by $T\to\cV\subset\sL_{n(P)}(\cW(\cX))$ and let $\widetilde{\cC}_P\subset|T|$ be the set of $t\in T$ where $\theta_{n(P)}(\gamma)$ has property $P$. Again, by Proposition \ref{prop:openness-geom-fibers-for-integrity}, Lemma \ref{l:S2-property-on-level0}(\ref{l:S2-property-on-level0::geom-pts}), and Proposition \ref{prop:R1-property}(\ref{l:R1::geom-pts}), we see $\rho(\widetilde{\cC}_P)=\cC_P\cap\cV$, and so by Chevalley's Theorem, it suffices to prove $\widetilde{\cC}_P$ is constructible. Constructibility of this locus is given by Proposition \ref{prop:openness-geom-fibers-for-integrity}, Lemma \ref{l:S2-property-on-level0}(\ref{l:S2-property-on-level0::constructible}), and Proposition \ref{prop:R1-property}(\ref{l:R1-level1::constructible}).

We next prove part (\ref{thm:main-can-cylinder::characterization}). By Serre's criterion for normality, $\cD$ is normal if and only if it is $(S_2)$ and $(R_1)$. By Lemma \ref{l:S2-property-on-level0}(\ref{l:S2-property-on-level0::DVR}) and Proposition \ref{prop:R1-property}(\ref{l:R1::DVR}), this is equivalent to the condition that $\theta_0(\gamma)$ is $(S_1)$ and $\theta_1(\gamma)$ is $(R_1)$-aspiring. Lastly, Proposition \ref{prop:openness-geom-fibers-for-integrity} shows that integrity is an open condition. Thus, $f$ has integrity if and only if $\theta_0(\gamma)$ does.
\end{proof}

Theorem \ref{thm:main-canonical-cylinder} now follows from Theorem \ref{thm:main-can-cylinder}.

\begin{proof}[{Proof of Theorem \ref{thm:main-canonical-cylinder}}]
Let $\gamma\in(\sL(\cX)\setminus\sL(\cX\setminus\cU))(k')$, i.e., $\gamma=(\pi\colon\cD\to D,f\colon\cD\to\cX)$ with $D=\Spec k'[[t]]$ and the generic fiber of $\pi$ maps to $\cU$. 
By Lemma \ref{l:ArtinarcsinCX-are-isos-overD0}, we see $\pi$ an isomorphism over the generic point of $D$. 
Then Theorem \ref{thm:main-can-cylinder}(\ref{thm:main-can-cylinder::characterization}) shows that $\gamma\in\cC'_\cX$ if and only if $\cD$ is normal and $f$ has integrity, which holds if and only if $\gamma\in\cC_\cX$.
\end{proof}

\section{The McKay correspondence:~proof of Theorems \ref{thm:mcvf-canonical-intro} and \ref{thm:reductive-McKay}}

We prove the following more general version of Theorem \ref{thm:mcvf-canonical-intro}.

\begin{theorem}\label{thm:mcvf-canonical}
Using Notation \ref{not:main-gms}, assume $Y$ is irreducible. If $f\colon\sL(Y) \to \Z \cup \{\infty\}$ be a measurable function such that $\bL^f$ is integrable on $\sL(Y)$. Then 

\begin{enumerate}[label=(\alph*)]


\item\label{thm:mcvf-canonical::b} We have $\bL^f$ is integrable on $\sL(Y)$ if and only if $\bL^{f \circ \sL(\cW(\pi)) - \het_{\cW(\cX)/Y}}$ is $\dim Y$-integrable on $\cC_\cX$.

\item\label{thm:mcvf-canonical::c} If $\bL^f$ is integrable on $\sL(Y)$, then
\[
	 \int_{\sL(Y)} \bL^{f} \diff\mu_Y = \int_{\cC_\cX} \bL^{f \circ \sL(\cW(\pi)) - \het_{\cW(\cX)/Y}} \diff\mu_{\cW(\cX), \dim Y}.
\]

\end{enumerate}
\end{theorem}
\begin{proof}
Theorems \ref{thm:main--A0-algebraic}, \ref{thm:val-crit}, and \ref{thm:main-canonical-cylinder} tell us that the hypotheses of \cite[Theorem 1.2]{SatrianoUsatine5} hold. The result follows.
\end{proof}

The next few results aid in the proof of our McKay correspondence.

\begin{proposition}\label{propositionDeterminantCotangentComplexRepresentation}
Let $G$ be a linearly reductive group over $k$, and let $V$ be a finite-dimensional representation of $G$. 
\begin{enumerate}

\item\label{propositionDeterminantCotangentComplexRepresentation::detL} Let $\pi: [V/G] \to [\Spec(k)/G]$ denote the quotient of the map $V \to \Spec(k)$, and let $\mathfrak{g}$ denote the adjoint representation of $G$. Then
\[
	\det L_{[V/G]} = \pi^*\left( \det V \otimes_k \det{\mathfrak{g}} \right).
\]

\item\label{propositionDeterminantCotangentComplexRepresentationQTrivial} If there exists some $n \in \Z_{>0}$ such that $(\det V)^{\otimes n}$ is the trivial representation, then there exists some $m \in \Z_{>0}$ such that
\[
	(\det L_{[V/G]})^{\otimes m} \cong \cO_{[V/G]}.
\]

\item\label{propositionDeterminantCotangentComplexRepresentation1Trivial} If $\det V$ is the trivial representation and $G$ is either connected or finite, then
\[
	\det L_{[V/G]} \cong \cO_{[V/G]}.
\]

\end{enumerate}

\end{proposition}

\begin{proof}
Let $\rho\colon V\to [V/G]=:\cX$ be the quotient map. 

From the exact triangle
\[
\rho^*L_\cX\to \Omega^1_V\to \g^\vee\otimes\cO_V,
\]
we see
\[
\rho^*(\det L_\cX)\simeq \det(\rho^*L_\cX)\simeq \omega_V\otimes\det \g.
\]
Furthermore, as $G$-representations, $\omega_V$ is isomorphic to the line bundle defined by $\det V$.
This therefore proves (\ref{propositionDeterminantCotangentComplexRepresentation::detL}).

In light of (\ref{propositionDeterminantCotangentComplexRepresentation::detL}), to prove (\ref{propositionDeterminantCotangentComplexRepresentation1Trivial}) we need only show that $\det \g$ is the trivial representation when $G$ is finite or connected. Since $\det \g$ is the determinant of the adjoint representation, the finite case is clear and the connected case follows from \cite[Corollary 8.31]{KnappBook}.

Lastly, to prove (\ref{propositionDeterminantCotangentComplexRepresentationQTrivial}), let $G^0$ be the connected component of the identity. Then the restriction of the adjoint representation of $G$ to $G^0$ is the adjoint representation of $G^0$, so it acts trivially by (\ref{propositionDeterminantCotangentComplexRepresentation1Trivial}). Therefore the $G$-action on $\g$ factors through $G/G^0$ which is finite, so we may take $m=n\cdot|G/G^0|$.
\end{proof}

\begin{proposition}\label{propGenericRepresentationSmallResolution}
Let $G$ be a linearly reductive group over $k$, and let $V$ be a generic representation of $G$. Then there exists an open subset $U \subset V/G$ such that if $\cU$ is the preimage of $U$ along $[V/G] \to V/G$, then $[V/G] \setminus \cU$ has codimension at least 2 in $[V/G]$ and $\cU \to U$ is an isomorphism.
\end{proposition}

\begin{proof}
Let $\cX:=[V/G]$ and $Y=V/G$. By \cite[Proposition 2.6]{ER}, the stable locus $\cX^{\s}\subset\cX$ is open, its image $Y^{\s}\subset Y$ is open, and $\cX^{\s}=\cX\times_Y Y^{\s}$. Since there exists a point with closed orbit, $\cX^{\s}$ is non-empty, hence dense. By \cite[Proposition 2.7]{ER}, $p^{\s}\colon\cX^{\s}\to Y^{\s}$ factors as a gerbe over a tame stack with coarse space $Y^{\s}$. Specifically, the proof of \cite[Proposition B.2]{ER} shows that the gerbe is the rigidification of $(I\cX^{\s})_0$ in the notation of \cite[Proposition B.6]{ER}. Since, by hypotheses, there is a stable point with trivial stabilizer, we see $(I\cX^{\s})_0\to\cX^{\s}$ is smooth with connected fibers of dimension $0$, hence trivial. Thus, $\cX^{\s}$ is a tame stack (which has finite inertia by definition) and $p^{\s}$ is a coarse space map, hence proper, quasi-finite, and a universal homeomorphism. Since $p^{\s}$ is proper, Theorems 2.2.5(2) and Corollary 2.2.7 of \cite{ConradKM} show there is a maximal open substack $\cU\subset\cX^{\s}$ which is an algebraic space and it is characterized by the property that its geometric points have trivial automorphism groups. Since there exists a stable point with trivial stabilizer, $\cU$ is non-empty. Since $p^{\s}$ is a homeomorphism, $U:=p^{\s}(\cU)$ is a non-empty open subset and $\cU=U\times_{Y^{\s}}\cX^{\s}$. It follows that $\cU\to U$ is a coarse space map and hence an isomorphism as $\cU$ is an algebraic space. Finally, $[V/G] \setminus \cU$ has codimension at least 2 in $[V/G]$ since $V$ is generic.
\end{proof}

\begin{proposition}
\label{prop:VmodG-Gorenstein-prop}
Let $G$ be a linearly reductive group over $k$, and let $V$ be a generic representation of $G$.
\begin{enumerate}

\item\label{prop:VmodG-Gorenstein-prop::det-torsion} If there exists some $n \in \Z_{>0}$ such that $(\det V)^{\otimes n}$ is the trivial representation, then $V/G$ is $\QQ$-Gorenstein.

\item\label{prop:VmodG-Gorenstein-prop::gorenstein-part} If $\det V$ is the trivial representation and $G$ is either connected or finite, then $V/G$ is Gorenstein. 

\end{enumerate}
\end{proposition}

\begin{proof}
Set $Y = V/G$ and $\cX = [V/G]$. We first note that $Y$ is normal by \cite[Proposition 6.15(b)]{HochsterRoberts}. Now let $U \subset Y$ and $\cU \subset \cX$ be as in the conclusion of \autoref{propGenericRepresentationSmallResolution}. We will show that $Y\setminus U$ has codimension at least 2 in $Y$. Suppose to the contrary that there exists a nonempty integral subscheme $D$ of $Y \setminus U$ that has codimension 1 in $Y$. Since $Y$ is normal, $D \cap Y^{\sm}$ is nonempty. Thus $D \cap Y^{\sm}$ is a nonempty Cartier divisor of $Y^{\sm}$, so the preimage of $D$ along $V \to V/G$ has codimension 1 in $V$, which contradicts the fact that $[V/G] \setminus \cU$ has codimension at least 2 in $[V/G]$. Therefore $Y\setminus U$ has codimension at least 2 in $Y$.

\begin{enumerate}

\item By \autoref{propositionDeterminantCotangentComplexRepresentation}(\ref{propositionDeterminantCotangentComplexRepresentationQTrivial}), there exists some $m \in \Z_{>0}$ such that $(\det L_{\cX})^{\otimes m} \cong \cO_\cX$. Since $U$ is isomorphic to the open substack $\cU$ of $\cX$, this implies that $\omega_U^{\otimes m} \cong \cO_U$. Since $Y\setminus U$ has codimension at least 2 in $Y$, this implies that $Y$ is $\QQ$-Gorenstein.

\item By \autoref{propositionDeterminantCotangentComplexRepresentation}(\ref{propositionDeterminantCotangentComplexRepresentation1Trivial}), $\det L_\cX \cong \cO_\cX$. Since $U$ is isomorphic to the open substack $\cU$ of $\cX$, this implies that $\omega_U \cong \cO_U$. Since $Y\setminus U$ has codimension at least 2 in $Y$, this implies that $Y$ is $1$-Gorenstein. Note that $Y$ is Cohen-Macaulay by \cite[Main Theorem]{HochsterRoberts}, so $Y$ is Gorenstein.\qedhere
\end{enumerate}
\end{proof}

\begin{remark}
Proposition \ref{prop:VmodG-Gorenstein-prop}(\ref{prop:VmodG-Gorenstein-prop::gorenstein-part}) also follows from \cite[Satz 2]{Knop1989}.
\end{remark}

Finally we turn to our McKay correspondence for linearly reductive groups. 

\begin{proof}[{Proof of Theorem \ref{thm:reductive-McKay}}]

By Proposition \ref{prop:VmodG-Gorenstein-prop}(\ref{prop:VmodG-Gorenstein-prop::det-torsion}), $V/G$ is $\QQ$-Gorenstein, so by \cite[Theorem B]{Schoutens2005}, $V/G$ is log-terminal. Therefore $\bL^{(1/m)\ord_{\cJ_{V/G,m}}}$ is integrable. The theorem now follows immediately from Theorem \ref{thm:mcvf-canonical-intro} and Proposition \ref{propGenericRepresentationSmallResolution}.
\end{proof}

\bibliographystyle{alpha}
\bibliography{MCVF-conj}

\end{document}